\documentclass[10pt, a4paper]{article}
\usepackage[T1]{fontenc}
\usepackage[a4paper, left=3.5cm, right=3.5cm, top=3.5cm]{geometry}
	
\usepackage[german, english]{babel}
\usepackage[utf8]{inputenc}
\usepackage{amssymb}
\usepackage{amsmath}
\usepackage{amsthm}
\usepackage{yfonts}
\usepackage{dsfont}
\usepackage{faktor}
\usepackage[pdftex]{graphicx}
\usepackage[arrow, matrix, curve]{xy}
\usepackage{nicefrac}
\usepackage{stmaryrd}
\usepackage{hyperref}
\usepackage[arrow, matrix, curve]{xy}
\usepackage{nth}
\usepackage{xcolor}
\usepackage{mathdots}

\newtheorem{num}{Nummerierung}[subsection]
\newtheorem{thm1}{Theorem}
\newtheorem{defin}[num]{Definition}
\newtheorem*{defin*}{Definition}
\newtheorem{prop}[num]{Proposition}

\newtheorem{lem}[num]{Lemma}

\newtheorem{cor}[num]{Corollary}

\newtheorem{rem}[num]{Remark}
\newtheorem*{rem*}{Remark}

\newtheorem{ex}[num]{Example}
\newtheorem{thm}[num]{Theorem}

\newtheorem*{conj*}{Conjecture}
\newtheorem*{thm*}{Theorem}
\newtheorem*{cor*}{Corollary}

\newcommand{\D}{\mathbb{D}}
\newcommand{\E}{\mathbb{E}}

\newcommand{\N}{\mathbb{N}}
\newcommand{\Z}{\mathbb{Z}}

\newcommand{\R}{\mathbb{R}}
\newcommand{\C}{\mathbb{C}}

\newcommand{\T}{\mathbb{T}}
\newcommand{\W}{\mathbb{W}}

\newcommand{\cB}{\mathcal{B}}

\newcommand{\cD}{\mathcal{D}}
\newcommand{\cE}{\mathcal{E}}

\newcommand{\cH}{\mathcal{H}}

\newcommand{\cP}{\mathcal{P}}
\newcommand{\cU}{\mathcal{U}}

\newcommand{\cO}{\mathcal{O}}
\newcommand{\cS}{\mathcal{S}}

\newcommand{\cW}{\mathcal{W}}

\newcommand{\g}{\mathfrak{g}}

\newcommand{\m}{\mathfrak{m}}
\newcommand{\n}{\mathfrak{n}}

\renewcommand{\a}{\mathfrak{a}}
\renewcommand{\t}{\mathfrak{t}}
\renewcommand{\k}{\mathfrak{k}}
\newcommand{\p}{\mathfrak{p}}

\newcommand{\z}{\mathfrak{z}}

\renewcommand{\u}{\mathfrak{u}}

\newcommand{\gl}{\mathfrak{gl}}

\newcommand{\spann}{\text{span}}

\newcommand{\tpitwo}{\tfrac{\pi}{2}}

\newcommand{\pitwo}{\frac{\pi}{2}}

\renewcommand{\Re}{{\rm Re}}
\renewcommand{\Im}{{\rm Im}}
\newcommand{\Ind}{{\rm Ind}}
\newcommand{\id}{{\rm id}}
\newcommand{\im}{{\rm im}}

\newcommand{\ev}{{\rm ev}}

\newcommand{\difftev}{\left.\frac{d}{dt}\right\vert_{t=0}}

\newcommand{\del}{\partial}
\newcommand{\tr}{{\rm tr}}
\newcommand{\ad}{{\rm ad}}
\newcommand{\Ad}{{\rm Ad}}
\newcommand{\tri}{\trianglelefteq}
\newcommand{\scp}{\langle\cdot,\cdot\rangle}
\newcommand{\SL}{{\rm SL}}
\newcommand{\GL}{{\rm GL}}

\newcommand{\Hom}{{\rm Hom}}
\newcommand{\End}{{\rm End}}

\newcommand{\qand}{\quad\text{and}\quad}
\newcommand{\qandthus}{\quad\text{and thus}\quad}
\newcommand{\qfor}{\quad\text{for}\quad}
\newcommand{\qforall}{\quad\text{for all}\quad}
\newcommand{\qsuchthat}{\quad\text{such that}\quad}
\newcommand{\wt}{\widetilde}

\newcommand{\hookto}{\hookrightarrow}
\newcommand{\epi}{\twoheadrightarrow}
\newcommand{\spec}{\text{spec}}

\def \wxi {\wt{\Xi}_{G_\C}}
\def \xig {\Xi_{G_\C}}

\def \KANC {K_\C A_\C N_\C}
\def \NAKC {N_\C A_\C K_\C}
\def \pr {{\rm pr}}
\def \temp {{\rm temp}}
\def \Inn {{\rm Inn}}

\title{Polynomial growth of holomorphic extensions of orbit maps of $K$-finite vectors at the boundary of the crown}
\author{Tobias Simon}
\begin{document}\selectlanguage{english}
\maketitle
\begin{abstract}
The Kr\"otz-Stanton Extension Theorem states that the orbit map of a K-finite vector in a Hilbert representation of a linear Lie group extends to a holomorphic map to a principal fibre bundle over the complex crown domain associated to the Riemannian symmetric space $G/K$. We extend this theorem to arbitrary connected semisimple Lie groups and prove polynomial growth estimates at the boundary. Using this, we show that the boundary values of these holomorphic extensions exist in the space of distribution vectors.
\end{abstract}
\tableofcontents
\section*{Introduction}
Let $G$ be a connected semisimple Lie group with global Cartan decomposition $G=K\exp(\p)$. For a Hilbert representation $(\pi,\cH)$, which is unitary when restricted to $K$, let $\cH^{[K]}\subseteq \cH$ be the subspace of \textit{$K$-finite vectors}, i.e. vectors $v\in \cH$ such that $$\dim_\C\spann_\C\pi(K)v<\infty.$$
If $(\pi,\cH)$ is additionally irreducible, or more generally has finite length, then $K$-finite vectors are analytic. Since any $K$-finite vector $v \in \cH^{[K]}$ is contained in a finite-dimensional $K$-invariant subspace, the orbit map $$\pi^v:K\to \cH,\quad k\mapsto \pi(k)v$$
extends to a holomorphic map $\pi^v:\wt{K}_\C\to \cH$, where $\k:=\textbf{L}(K)$ and $\wt{K}_\C$ is the simply connected Lie group with Lie algebra $\k_\C$. Suppose $G$ is linear, i.e. $G\subseteq G_\C$, where $G_\C$ is the universal complexification of $G$. One also knows how far one can extend in the $\p_\C$-direction. For the infinitesimal generator $\del\pi(x)$ such that $e^{t\del\pi(x)}=\pi(\exp(tx))$, one has
\begin{equation}\label{EqIncl}
\cH^{[K]}\subseteq \bigcap_{x\in \Omega_\p}\cD(e^{i\del\pi(x)}),\qfor \Omega_\p:=\left\{x\in \p\;|\;r_{\rm spec}(\ad(x))<\pitwo\right\},
\end{equation}
where $\cD(e^{i\del\pi(x)})$ denotes the domain of $e^{i\del\pi(x)}$ and $r_{\rm spec}$ is the spectral radius. This is an equivalent formulation of the \textit{Kr\"otz--Stanton Extension Theorem} (cf. \cite[Thm. 3.1]{KS04}). This theorem was originally proven for irreducible Banach representations of linear semisimple Lie groups but can be seen to generalize to smooth Fr\'{e}chet globalizations of moderate growth of Harish--Chandra modules of connected semisimple Lie groups with finite center. The same proof would apply if the \textit{Casselman--Wallach Globalization Theorem} (cf. \cite[Thm. 11.6.7]{Wa92}) were available without the assumption of finite the center.\\

For linear Lie groups, a different proof of the Globalization Theorem is given in \cite{BK14}, relying on lower bounds of matrix coefficients. As the Globalization Theorem is currently not available in general, we give an alternative proof of the Kr\"otz--Stanton Extension Theorem (cf. Theorem \ref{thmKrSt}), which uses the holomorphic extension of $K$-finite smooth eigenfunctions of the Casimir operator, as proven in \cite{KSch09} for linear connected semisimple Lie groups. In Appendix \ref{subsDiffEq}, we argue that this result extends to $(Z(\g),K)$-finite smooth functions on connected semisimple Lie groups. \\

In this paper, we study the boundary behaviour of $\|e^{i\del\pi(x)}v\|$ as $x\in \Omega_\p$ approaches the boundary $\del\Omega_\p$, for Hilbert globalizations $(\pi,\cH)$ of Harish--Chandra modules and $K$-finite vectors $v\in \cH^{[K]}$. We restrict ourselves to Hilbert globalizations instead of smooth Fr\'{e}chet globalizations of moderate growth since the topology on the latter is induced by a family of $G$-continuous Hilbert seminorms (cf. \cite{BK14}) and thus the context of Hilbert globalizations is sufficient for our purposes. The main focus of this paper is on upper estimates for the growth of the norm 
$$\|e^{it\del\pi(x_0)}v\|,\qfor x_0\in \del\Omega_\p\qand v\in \cH^{[K]},\quad\text{as}\quad t\nearrow 1.
$$
For irreducible unitary representations and $K$-fixed vectors, sharp results were obtained in \cite[Thm. 7.2]{KO08}, were it was shown that this growth is polylogarithmic in the sense that $$\|e^{it\del\pi(x_0)}v\|\asymp (1-t)^{-\alpha}\log^\beta(1-t),\quad\text{as}\quad t\nearrow 1,\qfor\alpha,\beta\geq 0.$$
Their method for $K$-fixed vectors seems to generalize to arbitrary $K$-finite vectors in unitary representations, as their results rely on the fact that the matrix coefficient $$\pi^{v,v}:G\to \C,\quad g\mapsto\langle v,\pi(g)v\rangle$$
satisfies an analytic differential equation with regular singularities (cf. \cite{CM82}), which is also the case for $K$-finite vectors. However, if the Hilbert representation $(\pi,\cH)$ is not unitary, the doubling trick (cf. \cite[Thm. 4.2]{KS04}) used in \cite{KO08} is no longer available, which reduces the determination of the growth of $\|e^{i\del\pi(x)}v\|$ to the growth of the holomorphic extension of $\pi^{v,v}$. To bypass this difficulty, we consider the holomorphic extension of the $K$-invariant but not $G$-invariant kernel
$$K^{v,v}:G\times G\to \C,\quad (g,h)\mapsto \langle \pi(g)v,\pi(h)v\rangle.
$$
The complex manifold to which this kernel extends is a holomorphic principal fibre bundle over $\Xi\times \Xi$, where $\Xi$ is the complex crown domain initially introduced in \cite{AG90}. Its close connections to representation theory and harmonic analysis have been explored in several the papers, including \cite{GK02a,GK02b,GKO04,KS04,KS05,KO08,KSch09}. For the universal complexification $G\to G_\R\subseteq G_\C$, the crown $\Xi$ is defined as 
\begin{equation}\label{EqCrown}
\Xi:=\Xi_{G_\C}/K_\C,\qfor K_\C:=\langle \exp(\k_\C)\rangle\subseteq K_\C\qand \Xi_{G_\C}:=G_\R\exp(i\Omega_\p)K_\C\subseteq G_\C.
\end{equation}
We refer to the universal covering $\wt{\Xi}_{G_\C}$ of $\Xi_{G_\C}$ as the \textit{principal crown bundle}. The Kr\"otz--Stanton Extension Theorem implies that, for a Hilbert globalization $(\pi,\cH)$ of a Harish--Chandra module and a $K$-finite vector $v\in \cH^{[K]}$, the orbit map $\pi^v:G\to \cH$ extends to a holomorphic map $\pi^v:\wxi\to \cH$. In particular, the kernel $K^{v,v}$ extends to a holomorphic kernel
$$K^{v,v}:\wxi\times \wxi\to \C\qsuchthat \|e^{i\del\pi(x)}v\|^2=K^{v,v}(\exp(-ix),\exp(ix)).
$$
The holomorphic extension of $K^{v,v}$ is equivalent to the holomorphic extension of $\pi^v$, which allows us to prove the Kr\"otz--Stanton Extension Theorem for arbitrary connected semisimple Lie groups using methods from \cite{KSch09} as mentioned above.\\
 
Before turning to our main results, we first recall basic properties of smooth vectors and distribution vectors, and discuss holomorphic functions on strips of moderate growth. A vector $v\in \cH$ is called \textit{smooth} if its orbit map $\pi^v:G\to \cH$ is smooth. The space of smooth vectors, denoted by $\cH^\infty$, carries a natural Fr\'{e}chet topology induced by the seminorms
\begin{equation}\label{EqDefDistr}
p_D:\cH^\infty\to [0,\infty),\quad v\mapsto \|d\pi(D)v\|,\qfor D\in \cU(\g),
\end{equation}
where $\cU(\g)$ denotes the universal enveloping algebra of $\g$. The elements of the corresponding antidual space $\cH^{-\infty}$ are called \textit{distribution vectors} and $\cH^{-\infty}$ contains $\cH$ as a subspace. In particular, one has $G$-equivariant linear embeddings
$$\cH^\infty\hookto \cH\hookto \cH^{-\infty}.
$$
Let $\pi:\R\to B(\cH)^\times$ be a strongly continuous one-parameter group and suppose $v\in \cH^\omega$ is an analytic vector, for which the orbit map $\pi^v:\R\to \cH$ extends to a holomorphic map
$$\pi^v:\cS_\beta:=\{z\in \C\;|\;|\Im(z)|<\beta\}\to \cH,\qsuchthat \|\pi^v(it)\|\leq C(\beta-|t|)^{-N},
$$
for $|t|<\beta$ and some fixed constants $C,N>0$. Then, the limit $$\lim_{t\nearrow \beta}\pi^{v}(it)\in \cH^{-\infty}$$
exists in the weak-$\ast$-topology (cf. Corollary \ref{corSlowGrowth}). This is a straightforward generalization of the well-known fact that holomorphic functions on a strip with moderate growth at the boundary have distributional boundary values. In \cite[Thm. 2.21]{FNO23}, the authors prove this for unitary representations using Laplace transforms. The following theorem verifies Hypothesis 1 and Hypothesis 2 in \cite{FNO23}.
\begin{thm1}\label{thm1}
Let $G$ be a connected semisimple Lie group and let $(\pi,\cH)$ be
\begin{itemize}
\item[{\rm (i)}] an arbitrary Hilbert globalization of a Harish--Chandra module, if $Z(G)$ is finite,
\item[{\rm (ii)}] a principal series representation or irreducible unitary representation, if $|Z(G)|=\infty$.
\end{itemize}
For $v\in \cH^{[K]}$, there exist $C,N>0$ such that
\begin{equation}\label{EqThm1}
\|e^{i\del\pi(x)}v\|\leq C\left(\pitwo -r_{\rm spec}(\ad(x))\right)^{-N},\qforall x\in \Omega_\p.
\end{equation}
In particular, $\lim_{t\nearrow 1}e^{it\del\pi(x)}v\in \cH^{-\infty}$ exists in the weak-$\ast$-topology, for all $x_0\in \del\Omega_\p$.
\end{thm1}
To prove this theorem, we use \textit{Casselman's Subrepresentation Theorem} (cf. \cite[Prop. 8.23]{CM82} for finite center and \cite[Appendix A.1]{FNO23} for general $G$), which implies the existence of an embedding of $(\g,K)$-modules $\cH^{[K]}\hookto \cH_\sigma^{[K]}$, where $\cH_\sigma^{[K]}$ are the $K$-finite vectors of a principal series representation associated to a finite-dimensional representation $(\sigma,W)$ of a minimal parabolic $P_{\rm min}=MAN$, which is unitary when restricted to $M$. Therefore, by applying the Globalization Theorem, it suffices to prove (\ref{EqThm1}) only for principal series representations. We do this by studying the asymptotic behaviour of the complexified Iwasawa component maps at the boundary of the complexified Iwasawa domain $\KANC$. To explain this technique, we recall the definition of principal series representations. Let $G=KAN$ be the Iwasawa composition corresponding to a maximal abelian subspace $\a\subseteq \p$, fix a positive system $\Sigma^+\subseteq \Sigma:=\Sigma(\g,\a)$ of restricted roots and let $\rho:=\tfrac12\sum_{\gamma\in \Sigma^+}m_\gamma\gamma$ be the half-sum of all positive roots with multiplicities $m_\gamma:=\dim_\R\g_\gamma$. Let
$$\kappa:G\to K,\quad \alpha:G\to A,\quad H:=\log\alpha:G\to \a\qand \eta:G\to N
$$
be the corresponding Iwasawa component maps. We also define $\beta:G\to B:=AN,\; g\mapsto \alpha(g)\eta(g)$ and $\exp(H)^\mu:=e^{\mu(H)}$, for $\mu\in \a_\C^\ast$ and $H\in \a$. For $M=Z_K(\a)$ and a  finite-dimensional representation $(\sigma,W)$ of a minimal parabolic $P_{\rm min}=MAN$, which is unitary when restricted to $M$, the principal series representation $\Ind_{P_{\rm min}}^G(\sigma)$ can be realized in the compact picture on the Hilbert space $$\cH_\sigma:=L^2(K\times_{\sigma}M)=\left\{f:K\to W\;|\;f(km)=\sigma(m)^{-1}f(k),\quad \int_{K/M}\|f(k)\|^2dk<\infty\right\}$$
of square-integrable sections of the associated vector bundle $K\times_{\sigma}W\epi K/M$. For $\rho\in \a^\ast$ as above and $\sigma_\rho(man):=a^\rho\sigma(man)$, the principal series representation $\pi_\sigma$ acts by
\begin{equation}\label{defPrinc1}
(\pi_\sigma(g)f)(k):=\sigma_\rho(\beta(g^{-1}k))^{-1}f(\kappa(g^{-1}k)), \qfor g\in G,\quad k\in K,\quad f\in \cH_\sigma.
\end{equation}
It follows from \cite[Thm. 1.8]{KS04} that there exist \textit{complexified Iwasawa component maps} 
\begin{align*}
&\wxi\times K\to \wt{K}_\C, \quad(x,k)\mapsto \kappa(x^{-1}k),\\
&\wxi\times K\to \a_\C, \;\quad(x,k)\mapsto H(x^{-1}k),\\
&\wxi\times K\to N_\C, \quad(x,k)\mapsto \eta(x^{-1}k),
\end{align*}
extending the corresponding maps $G\times K\to K\times \a\times N$. These maps are holomorphic in the first variable and analytic in the second variable. We control the growth of the holomorphic extensions of the orbit map of $K$-finite vectors in principal series representations by studying the growth of the complexified Iwasawa component maps. A suitable notion of growth is defined as follows. Consider the maximal compact subgroup $U:=\langle \exp_{G_\C}(\u)\rangle$ of $G_\C$, for $\u:=\k+i\p$, and the \textit{maximal scale function} of $G_\C$ (cf. \cite{BK14}) given by $$s_{\rm max}^{G_\C}: G_\C\to [1,\infty),\quad u\exp(iX)\mapsto e^{r_{\rm spec}(\ad(iX))},\qfor u\in U,\quad X\in \u.
$$
Our main result, from which we derive Theorem \ref{thm1}, is the following (cf. Theorem \ref{ThmMainResult}):
\begin{thm1}\label{thm2}
Let $G$ be a connected semisimple Lie group with not necessarily finite center. Then there exists $N>0$ such that the functions
\begin{align*}
{\rm(i)}\;&\Omega_\p\ni x\mapsto \sup_{k\in K} s_{\rm max}^{G_\C}(\kappa(\exp(-ix)k)),\\
{\rm(ii)}\;&\Omega_\p\ni x\mapsto \sup_{k\in K}s_{\rm max}^{G_\C}(\alpha(\exp(-ix)k)),\\
{\rm(iii)}\;&\Omega_\p\ni x\mapsto \sup_{k\in K}s_{\rm max}^{G_\C}(\eta(\exp(-ix)k)),
\end{align*}
are bounded by a constant multiple of $\left(\pitwo-r_{\rm spec}(\ad(x))\right)^{-N}$.
\end{thm1}
We now sketch the application of Theorem \ref{thm2} to principal series representations. Let $(\sigma_0,W)$ be a unitary representation of $M$, let $\lambda\in \a_\C^\ast$ be a linear functional and consider the finite-dimensional representation $$\sigma:=\sigma_0\times \lambda\times 1:P_{\rm min}=MAN,\quad man\mapsto \sigma_0(m)a^\lambda.$$
For a $K$-finite vector $f\in \cH_\sigma^{[K]}$ in the corresponding principal series representation, one has
\begin{align*}
\|e^{i\del\pi_\sigma(x)}f\|^2=&\int_{K/M}\left|\alpha(x^{-1}k)^{-(\lambda+\rho)}\right|^2\left\| f(\kappa(g^{-1}k))\right\|^2 dkM\\
\leq & \sup_{k\in K}\left|\alpha(x^{-1}k)^{-(\lambda+\rho)}\right|^2\cdot\sup_{k\in K}\left\| f(\kappa(g^{-1}k))\right\|^2
\end{align*}
for all $x\in \Omega_\p$. In particular, Theorem \ref{thm2} implies that there exist $C,N>0$ such that
$$\|e^{i\del\pi_\sigma(x)}f\|^2\leq C\left(\pitwo-r_{\rm spec}(\ad(x))\right)^{-N}.
$$
In \cite{KSch09}, the authors showed that every $(Z(\g),K)$-finite function $f:G\to \C$ extends to a holomorphic function $f:\wxi\to \C$. With a similar argument as above, we were able to show the following theorem (cf. Corollary \ref{corModGrowth}):
\begin{thm1}\label{thm3}
Let $G$ be a connected semisimple Lie group with finite center and let $f:G\to \C$ be a $(Z(\g),K)$-finite smooth function of moderate growth, i.e. there exists $d>0$ such that
$$
\sup_{g\in G}\|g\|^{-d}\cdot|L_{D}(f)(g)|<\infty,\qforall D\in \cU(\g)\qand \|g\|:=s_{\rm max}^{G}(g).
$$
Then, there exist $C,N>0$ such that the holomorphic extension $f:\wxi\to \C$ satisfies
$$|f(\exp(ix))|\leq C\cdot \left(\pitwo-r_{\rm spec}(\ad(x))\right)^{-N},\qforall x\in \Omega_\p.
$$
\end{thm1}
\textbf{Acknowledgements:} We are very grateful to Karl-Hermann Neeb for his supervision and the many fruitful discussions. We thank Ricardo Correa da Silva for always lending us his ear and Bart van Steirteghem for his advice on the algebro geometric aspects. We are also very grateful to Gestur \'{O}lafsson for his suggestion that our methods of the growth of the complexified Iwasawa components also apply to Poisson transforms of distributional sections. This research was supported by DFG-grant NE 413/10-2.
\section*{Notation}\label{Notation}
\begin{itemize}
\item $G$ denotes a connected semisimple Lie group with Cartan decomposition $G=K\exp(\p)$.
\item Let $\a\subseteq \p$ be a maximal abelian subspace and let $\Sigma^+\subseteq \Sigma:=\Sigma(\g,\a)$ be a fixed positive system of restricted roots.
\item Let $\g=\k\oplus \a\oplus \n$ be the associated Iwasawa decomposition and let $G=KAN$ be the corresponding global Iwasawa decomposition.
\item Let $\m=\z_\k(\a)$ and $M=Z_K(\a)$. We fix a minimal parabolic subgroup $P_{\rm min}:=MAN$.
\item For the universal complexification $\eta_G:G\to G_\C$, we denote with $$G_\R,K_\R,K_\C,A_\C  \qand N_\C$$
the integral subgroups of $G_\C$ associated to the subalgebras $\g,\k,\k_\C,\a_\C,\n_\C$ of $\g_\C$.
\item We usually denote elements in $\Omega_\p$ with $x$ and elements in $\del\Omega_\p$ with $x_0$. 
\end{itemize}
\newpage
\section{Hilbert representations of semisimple Lie groups}\label{secKrSt}
In this section, we use results from \cite{KSch09}, which are discussed in Appendix \ref{subsDiffEq}, to generalize the Kr\"otz--Stanton Extension Theorem (cf. \cite[Thm 3.1]{KS04}) to connected semisimple Lie groups $G$ with not necessarily finite center. Throughout, $G$ denotes a connected semisimple Lie group with Lie algebra $\g$ and global Cartan decomposition $G=K\exp(\p)$.
\subsection{Harish Chandra modules and Fr\'{e}chet globalizations}\label{subsHC}
This subsection is a brief introduction to the theory of $(\g,K)$-modules and globalizations of $(\g,K)$-modules. References for the algebraic theory are \cite[\S 3]{Wa88} and \cite[\S 4]{BK14} and globalizations are discussed in \cite[\S 11]{Wa92} and \cite{BK14}. 
\begin{defin}\label{gk}
\rm{A $(\g,K)$-module $V$ is a complex vector space carrying two representations $$d\pi:\cU(\g)\to \End(V)\qand \pi:K\to \GL(V)$$
\begin{itemize}
\item[(1)] such that every vector $v\in V$ is continuous and $K$-finite, i.e contained in a finite-dimensional $K$-invariant subspace $W$ on which $(\pi,W)$ is \textbf{unitarizable} and continuous,
\item[(2)] for $D\in \cU(\g)$, $Y \in \k$ and $k\in K$, one has
$$\difftev \pi(\exp(tY))v=d\pi(Y)v\qand \pi(k)d\pi(D)\pi(k)^{-1}=d\pi(\Ad(k)D).$$
\end{itemize}
A $(\g,K)$-module $V$ is called
\begin{itemize}
\item[(i)] \textit{finitely generated}, if $V$ is finitely generated as a $\cU(\g)$-module.
\item[(ii)] \textit{$Z(\g)$-finite}, if $\dim_\C d\pi(Z(\g))v<\infty$, for all $v\in V$.
\item[(iii)] \textit{admissible}, if the isotypical components $V_{[\tau]}\subseteq V$ are finite-dimensional for all irreducible unitary representations $(\tau,W_\tau)$, where $V_{[\tau]}$ is defined as the image of the map
$$W_\tau \otimes \Hom_K(W_\tau,V)\to V,\quad w\otimes T\mapsto T(w).
$$
\item[(iv)] a \textit{Harish--Chandra module}, if $V$ is finitely generated and admissible.
\end{itemize}
}
\end{defin}
\begin{rem}\label{remAdm}
\rm{(a) Equivalently a $(\g,K)$-module $V$ is a $\cU(\g)$-module such that the $\cU(\k)$-module structure integrates to $K$ and every vector $v\in V$ is \textit{$K$-finite}, i.e. its $K$-orbit spans a finite-dimensional subspace. Since every vector in $V$ is $K$-finite, the $K$-module $V$ decomposes as a sum of isotypical components $V_{[\tau]}$, for $[\tau]\in \widehat{K}$, i.e. $V=\sum_{[\tau]\in \widehat{K}}V_{[\tau]}$.\\

(b) Usually $(\g,K)$-modules are considered in the setting of reductive Lie groups of Harish--Chandra class, i.e. finite coverings of linear reductive Lie groups for which $\Ad(G)\subseteq \text{Inn}(\g_\C)$. For connected reductive Lie groups, $\Ad(G)\subseteq \Inn(\g_\C)$ is automatic. We do \textbf{not} limit ourselves to finite coverings since this excludes simply connected hermitian Lie groups such as the universal covering of $\SL(2,\R)$. If $G$ is simple and $Z(G)$ is not finite, then $G$ is a simply connected hermitian Lie group, $$K\cong \R\times [K,K]\qand Z(G)/(Z(G)\cap [K,K])\cong \Z.$$
In particular, not all finite-dimensional representations of $K$ are unitarizable, which explains why we additionally assume this in Definition \ref{gk}(1). From the perspective of unitary representation theory, where one studies unitary representations of $G$ and considers the $(\g,K)$-module $\cH^{[K]}$ of $K$-finite smooth vectors in $\cH$, this is the natural context.\\

(c) A finitely generated, $Z(\g)$-finite $(\g,K)$-module $V$ is a finitely generated $\cU(\n)$-module. This is a consequence of Osborne's Lemma (cf. \cite[Prop. 3.7.1]{Wa88}), which states that there exists a finite-dimensional subspace $F\subseteq \cU(\g)$ such that $$\cU(\g)=\cU(\n)FZ(\g)\cU(\k).$$
If $W\subseteq V$ is a finite-dimensional subspace which generates $V$ as a $\cU(\g)$-module, then the $(Z(\g),K)$-finiteness of $V$ implies that $$d\pi\left(FZ(\g)\cU(\k)\right)W\subseteq V$$
is a finite-dimensional subspace which generates $V$ as a $\cU(\n)$-module. A $(\g,K)$-module, which is finitely-generated as a $\cU(\n)$-module, is admissible (cf. \cite[Thm. 4.2.6]{Wa88}).}\\

(d) The \textit{Casselman Subrepresentation Theorem} implies that every Harish--Chandra module embeds as a $(\g,K)$-module in the $(\g,K)$-module of $K$-finite vectors of a principal series representation as defined in (\ref{defPrinc1}) (cf. \cite[Thm. 3.8.3]{Wa88} or \cite[Prop. 8.23]{CM82} for finite center and \cite[Appendix A.1]{FNO23} for arbitrary connected semisimple Lie groups). This allows one to reduce many problems about Harish--Chandra modules to principal series representations. The reason why we assume that the finite-dimensional $K$-subrepresentations are unitarizable is that the $K$-finite vectors in principal series representations have this property and thus the Subrepresentation Theorem can only hold for such Harish-Chandra modules.
\end{rem}
Having discussed the algebraic objects, we now discuss their globalizations. For a more detailed discussion, we refer to \cite{BK14}.
\begin{defin}
\rm{Let $(\pi,F)$ be a Fr\'{e}chet representation of $G$, i.e. a representation of $G$ on a Fr\'{e}chet space $F$ such that the action $G\times F\to F$ is continuous. A vector $v\in F$ is called \textit{smooth} if its orbit map $\pi^v:G\to F$ is smooth. Let
$$d\pi:\cU(\g)\to \End(F^\infty)
$$
be the corresponding derived representation. Let $(p_n)_{n\in \N}$ be a family of continuous seminorms of $F$ inducing the topology. The space of smooth vectors, denoted by $F^\infty$, carries a natural Fr\'{e}chet topology induced by the seminorms
$$
p_{n,D}:\cH^\infty\to [0,\infty),\quad v\mapsto p_n(\pi(D)v),\qfor D\in \cU(\g)\qand n\in \N.
$$
The elements of the corresponding antidual space $F^{-\infty}$ are called \textit{distribution vectors}.
}
\end{defin}
\begin{defin}
\rm{Let $(\pi,F)$ be a Fr\'{e}chet representation.
\begin{itemize}
\item[(i)] $(\pi,F)$ is called \textit{smooth}, if every vector in $F$ is smooth.
\item[(ii)] The maximal scale function $\|\cdot\|:G\to [1,\infty)$ is defined as $\|k\exp(X)\|=e^{\|X\|}$, for $k\in K$, $X\in \p$ and a $K$-invariant norm $\|\cdot\|$ on $\p$.
\item[(iii)] $(\pi,F)$ is called of \textit{moderate growth} if for every continuous seminorm $p:F\to [0,\infty)$, there exists $d>0$ and a continuous seminorm $q:F\to [0,\infty)$ such that
$$p(\pi(g)v)\leq \|g\|^dq(v),\qforall v\in F,\quad g\in G.
$$
\item[(iv)] A seminorm $p:F\to [0,\infty)$ is called \textit{$G$-continuous} if the $G$-action $G\times (F,p)\to (F,p)$ is continuous.
\end{itemize}
}
\end{defin}
\begin{defin}
\rm{Let $V$ be a $(\g,K)$-module. A Fr\'{e}chet representation $(\pi,F)$ is called a \textit{Fr\'{e}chet globalization} of $V$ if the space of $K$-finite vectors $F^{[K]}$ consists of smooth vectors and is isomorphic to $V$ as a $(\g,K)$-module.
}
\end{defin}
\begin{thm}\label{ThmHilbSemi}
{\rm (\cite[Thm. 5.5, Cor. 5.6]{BK14})} Let $G$ be a connected semisimple Lie group and let $(\pi,F)$ be a smooth Fr\'{e}chet globalization of moderate growth of a Harish--Chandra module $V$. Then, there exists a countable family of $G$-continuous $K$-invariant Hilbert semi-norms inducing the topology on $F$.
\end{thm}
\begin{proof}
The proof of \cite[Thm. 5.5, Cor. 5.6]{BK14}) relies on the estimate $\dim_\C V_{[\tau]}\lesssim(1+d_\tau)^N$, for $[\tau]\in \widehat{K}$ and the dimension $d_\tau\in \N_0$ of $\tau$. This estimate also holds for Harish--Chandra modules of arbitrary semisimple Lie groups since the Casselman Subrepresentation Theorem holds (cf. \cite[\S 6.1]{FNO23}) and thus Frobenius reciprocity is applicable. The remaining arguments generalize straightforwardly.
\end{proof}
\begin{rem}\label{remFrechGlob}
\rm{Let $(\pi,F)$ be a smooth Fr\'{e}chet globalization of moderate growth of a Harish--Chandra module and let $(\scp_n)_{n\in \N}$ be a family of $K$-invariant positive-semidefinite sesquilinear forms such that the corresponding Hilbert seminorms $\|\cdot\|_n$ are $G$-continuous and induce the topology on $F$. Let 
$$F_n:=\{v\in F\;|\; \|v\|_n=0\}\qand \cH_n:=\overline{F/F_n}^{\|\cdot \|_n}
$$
be the corresponding completion. Equip $\cH:=\prod_{n\in \N}\cH_n$ with the product topology and consider the embedding of Fr\'{e}chet representations $$\Phi:F\hookto \cH,\quad v\mapsto (v\;\text{mod}\;F_n)_{n\in\N}. $$
Since the seminorms $\|\cdot \|_n$ induce the topology on $F$ and $F$ is complete, $\Phi$ is an injective map between Fr\'{e}chet spaces with closed range, i.e. is a homeomorphism onto its image by the open mapping theorem.}
\end{rem}
For connected semisimple Lie groups with finite center, one has the Casselman--Wallach Globalization Theorem (cf. \cite[Thm. 11.6.7]{Wa92}).
\begin{thm}
Let $G$ be a connected semisimple Lie group with finite center and let $(\pi_i,F_i)$ be two smooth Fr\'{e}chet globalizations of a Harish--Chandra module $V$ of moderate growth. Then, the corresponding isomorphism $F_1^{[K]}\to F_2^{[K]}$ of $(\g,K)$-modules extends to a continuous $G$-equivariant isomorphism $F_1\to F_2$ of Fr\'{e}chet representations.
\end{thm}
\subsection{The Kr\"otz--Stanton Extension Theorem}\label{subsKS}
We now present an alternative proof of the \textit{Kr\"otz--Stanton Extension Theorem} that applies to arbitrary connected semisimple Lie groups $G$. The original proof heavily relies on the Casselman--Wallach Globalization Theorem and our proof uses the holomorphic extension of matrix coefficients. Throughout this subsection, we assume that $G$ is \textbf{simply connected}. 
\begin{defin}
\rm{The \textit{complex crown domain} is defined as
$$\Xi:=G\times_K i\Omega_\p,\qfor \Omega_\p:=\left\{x\in \p\;|\;r_{\rm spec}(\ad(x))<\pitwo\right\},
$$
and it carries a complex structure induced from the embedding
$$\Xi\cong G_\R\exp(i\Omega_\p).o\subseteq G_\C/K_\C,\qfor o:=eK_\C\in G_\C/K_\C.
$$
The \textit{principal crown bundle} $\wxi$ is defined as the universal covering of the $G$-left invariant holomorphic $K_\C$-principal fibre bundle $$\Xi_{G_\C}=G_\R\exp_{G_\C}(i\Omega_\p)K_\C\subseteq G_\C.$$
}
\end{defin}
\begin{rem}
\rm{We recall some facts about $\wxi$ from \cite[\S 3]{FNO23}. Let $q_{\xig}:\wxi\to \xig
$ be the universal covering. For each $\wt{e}\in q_{\xig}^{-1}(\{e\}),$ there exists a unique lift $$\exp:i\Omega_\p \to \wxi\quad\text{of}\quad \exp:i\Omega_\p\to \xig\qsuchthat\exp(0)=\wt{e}.$$
Furthermore, there exists a $G$-left action and $\wt{K}_\C$-right action on $\wxi$ such that the universal covering becomes equivariant. In particular, $\wxi$ is a holomorphic $\wt{K}_\C$-principal fibre bundle over $\Xi$ and
$$\wxi=G\exp(i\Omega_\p)\wt{K}_\C\cong (G\times i\Omega_\p)\times_K \wt{K}_\C.$$
Throughout this paper, we identify $G$ with the totally real submanifold $G.\wt{e}\subseteq \wxi$.}
\end{rem}
\begin{defin}
\rm{We say that a smooth function $f:G\to \C$ is \textit{$(Z(\g),K)$-right finite}, if the function $f$ is $K$-right finite, i.e. $\spann_\C \{R_k f\;|\;k\in K\}\subseteq C^\infty(G)$ is finite-dimensional, for the right translates $R_kf(g):=f(gk)$, and $f$ is $Z(\g)$-finite.}
\end{defin}
\begin{defin}
\rm{For a Hilbert representation $(\pi,\cH)$ of a connected semisimple Lie group $G$ which is unitary on $K$, the dual representation $(\pi^\ast,\cH)$ is defined as $\pi^\ast(g):=\pi(g^{-1})^\ast$. Note that $(\pi^\ast,\cH)$ is also a Hilbert representation, which is unitary when restricted to $K$.
}
\end{defin}
The following proposition is the main tool in our proof of the Kr\"otz-Stanton Extension Theorem. Its proof heavily relies on the results from \cite{KSch09}, which generalize to arbitrary connected semisimple Lie groups (cf. Appendix \ref{subsDiffEq}).
\begin{prop}\label{PropMatrixCoeff}
Let $(\pi,\cH)$ be a Hilbert representation of a connected semisimple Lie group $G$ which is unitary on $K$ and let $v,w\in \cH^{[K]}$ be $Z(\g)$-finite vectors. Then, the kernels
\begin{align*}
K^{v,w}:G\times G\to \C,&\quad\quad(g_1,g_2)\mapsto \langle \pi(g_1)v,\pi(g_2)w\rangle\\
\qand K_{\rm inv}^{v,w}:G\times G\to \C,&\quad\quad(g_1,g_2)\mapsto \langle \pi^\ast(g_1) v,\pi(g_2)w\rangle
\end{align*}
are $(Z(\g\oplus \g),K\times K)$-right finite and extend to holomorphic kernels $$K^{v,w}:\wt{\Xi}_{G_\C}\times \wt{\Xi}_{G_\C}\to \C\qand K_{\rm inv}^{v,w}:\wt{\Xi}_{G_\C}\times \wt{\Xi}_{G_\C}\to \C.$$
The kernel $K_{\rm inv}^{v,w}$ and its holomorphic extension are $G$-invariant.
\end{prop}
\begin{proof}
Since $v\in \cH^{[K]}$ is $(Z(\g),K)$-finite, it generates a finitely generated $Z(\g)$-finite $(\g,K)$-submodule $V$ of $\cH^{[K]}$ which is admissible by Remark \ref{remAdm}(c). Hence, $V$ consists of analytic vectors (cf. \cite[Cor. 5.6]{GKL12}) and the closure of $V$ in $\cH$ is $G$-invariant. We may thus assume that $V=\cH^{[K]}$ is a finitely generated admissible $(\g,K)$-module. It follows from \cite[Lemma 4.3.2]{Wa88} that $\cH^{[K]}$ is a finitely generated admissible $(\g,K)$-module for $\pi^\ast$ as well. In particular, $$I:=\ker(d\pi)\cap Z(\g)\tri Z(\g)\qand I^\ast:=\ker(d\pi^\ast)\cap Z(\g)\tri Z(\g)$$
are ideals of finite codimension. Since $Z(\g\oplus \g)=Z(\g)\otimes Z(\g)$, it follows that $$I\otimes Z(\g)+Z(\g)\otimes I\tri Z(\g\oplus\g)\qand I^\ast\otimes Z(\g)+Z(\g)\otimes I\tri Z(\g\oplus \g)$$
are ideals in $Z(\g\oplus \g)$ of finite codimension which annihilate $K^{v,w}$ and $K_{\rm inv}^{v,w}$ respectively. Therefore, $K^{v,w}$ and $K_{\rm inv}^{v,w}$ are $(Z(\g\oplus \g),K\times K)$-finite smooth functions. For the connected semisimple Lie group $G\times G$, one has $$(G\times G)_\C\cong G_\C\times G_\C\qand \wt{\Xi}_{(G\times G)_\C}\cong\wt{\Xi}_{G_\C}\times \wt{\Xi}_{G_\C}.$$
Since $K^{v,w}$ and $K_{\rm inv}^{v,w}$ are $(Z(\g\oplus \g),K\times K)$-finite, Theorem \ref{ThmHolExt} implies that these functions extend to holomorphic functions $$K^{v,w},K_{\rm inv}^{v,w}:\wt{\Xi}_{G_\C}\times \wt{\Xi}_{G_\C}\to \C.$$
The holomorphic extension of $K_{\rm inv}^{v,w}$ is $G$-invariant by the uniqueness of analytic continuation since $K_{\rm inv}^{v,w}:G\times G\to \C$ is $G$-invariant.
\end{proof}
\begin{cor}\label{corMatrixCoeff}
Let $(\pi,\cH)$ be a Hilbert representation of a connected semisimple Lie group $G$ which is unitary on $K$ and let $v,w\in \cH^{[K]}$ be $Z(\g)$-finite vectors. Then, the matrix coefficient 
$$\pi^{v,w}:A\to \C\quad a\mapsto \langle v,\pi(a)w\rangle$$ 
extends to a holomorphic map $\pi^{v,w}:A\exp(2i\Omega_\a)\to \C$.
\end{cor}
\begin{proof}
We identify $A\exp(i\Omega_\a)$ with $A\,\exp(i\Omega_\a)\subseteq\wt{\Xi}_{G_\C}$. The kernel $$K_{\rm inv}^{v,w}:A\exp(i\Omega_\a)\times A\exp(i\Omega_\a)\to \C$$
is $A$-invariant and holomorphic. Let $H_i,H_i^\prime \in \Omega_\a$, for $i=1,2$, such that $H_1-H_2=H_1^\prime-H_2^\prime$. Since $\Omega_\a$ is convex and open, there exists $\varepsilon>0$ such that
$$H_i(t):=tH_i+(1-t)H_i^\prime\in \Omega_\a,\qforall -\varepsilon<t<1+\varepsilon.
$$
Consider the holomorphic map $f:\{z\in \C\;|\;-\varepsilon<\Im(z)<1+\varepsilon\}\to\C$,
$$f(z):=K^{v,w}(a_1\exp(H_1(z)),a_2\exp(H_2(z))),\qfor a_1,a_2\in A.
$$
In particular, $H_1-H_1^\prime=H_2-H_2^\prime$ implies $H_i(t)=H_i^\prime+t(H_1-H_2^\prime)$ and thus
\begin{align*}
f(t)=K_{\rm inv}^{v,w}(a_1\exp(H_1(t)),a_2\exp(H_2(t)))=K_{\rm inv}^{v,w}(a_1\exp(H_1^\prime),a_2\exp(H_2^\prime))
\end{align*}
by $A$-invariance. This implies that $f$ is constant on $\R$. Therefore $f$ is constant and thus it follows from $f(0)=f(i)$ that $K_{\rm inv}^{v,w}$ only depends on $$a_1a_2^{-1}\exp(i(H_1-H_2))\in A\exp(2i\Omega_\a)\cong \a+2i\Omega_\a.$$
Thus, the restriction $K_{\rm inv}^{v,w}:A\exp(i\Omega_\a)\times A\exp(i\Omega_\a)\to \C$ factors through
$$m:A\exp(i\Omega_\a)\times A\exp(i\Omega_\a)\to A\exp(2i\Omega_\a),\quad (a_1,a_2)\mapsto a_1^{-1}a_2$$
and since $K_{\rm inv}^{v,w}:A\times A\to \C$ factors through $\pi^{v,w}$, the assertion follows.
\end{proof}
\begin{lem}\label{LemSelfAdj}
Let $\pi:\R\to B(\cH)^\times$ be a strongly continuous Hilbert representation of $\R$ and $v\in \cH^\infty$. Suppose there exists $\varepsilon>0$ such that the map
$$f:\R\times \R\to \C,\quad (s,t)\mapsto \langle \pi(s)v,\pi(t)v\rangle
$$
extends to a holomorphic map 
$$f:\cS_{\pm \varepsilon}\times \cS_{\pm \varepsilon}\to \C,\qfor \cS_{\pm \varepsilon}:=\{z\in \C\;|\;|\Im(z)|<\varepsilon\}.$$
Then, the orbit map $\pi^v:\R\to \cH$ extends to a holomorphic map $\pi^v:\cS_{\pm \varepsilon}\to \cH$.
\end{lem}
\begin{proof}
Let $A:\cD(A)\to \cH$ be the infinitesimal generator of $\pi$. Since $f:\cS_{\pm \varepsilon}\times \cS_{\pm \varepsilon}\to \C$ is a holomorphic function, it coincides with its Taylor series on the polydisc $B_\varepsilon(0)\times B_\varepsilon(0)$ and thus
$$f(s,t)=\sum_{n,m=0}^\infty \frac{s^nt^m}{n!m!}\langle A^nv,A^mv\rangle,\qforall |s|,|t|<\varepsilon.
$$
The Cauchy-Hadamard Theorem for power series on polydiscs (cf. \cite[Chapter I, \S 3, Thm. 4]{Sh92}) implies
\begin{align*}
\frac{1}{\varepsilon}\geq\limsup_{n+m\to \infty}\left( \frac{1}{n!m!}|\langle A^nv,A^mv\rangle|\right)^{1/(n+m)}\overset{n=m}{\geq}\limsup_{n\to \infty}\left( \frac{1}{n!}\|A^nv\|\right)^{1/n}
\end{align*}
and thus it follows that $r\geq \varepsilon$, where $r$ is the radius of convergence of the power series $e^{tA}v=\sum_{k=0}^\infty \frac{t^k}{k!}A^kv$. Therefore the assertion follows.
\end{proof}
If $(\pi,\cH)$ is a Hilbert representation and $v\in \cH^{[K]}$ is a $K$-finite vector, it is contained in a finite-dimensional $K$-invariant subspace $\cE\subseteq \cH^{[K]}$. In particular, the restriction of $\pi\vert_K$ to $\cE$ extends to a holomorphic representation of $\wt{K}_\C$ since the $\k_\C$-representation on $\cE$ integrates. Therefore, the orbit map 
$$\pi^v:K\to \cH,\quad k\mapsto \pi(k)v$$
extends to a holomorphic map $\pi^v:\wt{K}_\C\to \cE\subseteq \cH^{[K]}$. The following theorem was originally proven for irreducible Banach representations of linear simple Lie groups in \cite{KS04}. Our contribution is removing the linearity condition and providing a new proof. 
\begin{thm}\label{thmKrSt}{\rm~\textbf{(The Kr\"otz--Stanton Extension Theorem)}}\\
Let $G$ be a connected semisimple Lie group and let $(\pi,\cH)$ be a Hilbert representation of $G$. For a $Z(\g)$-finite vector $v\in \cH^{[K]}$, the orbit map $\pi^v:G\to \cH$ extends to a holomorphic map
$$\pi^v:\wt{\Xi}_{G_\C}\to \cH,\quad g\exp(ix)k\mapsto \pi(g)e^{i\del\pi(x)}\pi^v(k).$$
\end{thm}
\begin{proof}
For $h\in \del\Omega_\a$, the map $$\R\times \R\to \C,\quad (s,t)\mapsto  K^{v,v}(\exp(sh),\exp(th))=\langle\pi(\exp(sh))v,\pi(\exp(th))v\rangle$$ 
extends to a holomorphic map $\cS_{\pm 1}\times \cS_{\pm 1}\to \C$ by Proposition \ref{PropMatrixCoeff}. In particular, Lemma \ref{LemSelfAdj} implies that $$v\in \cD(e^{it\del\pi(h)}),\qforall |t|<1.$$
Furthermore, $\Omega_\a:=\a\cap \Omega_\p=[0,1)\del\Omega_\a$ and $\Omega_\p=\Ad(K)\Omega_\a$ imply $v\in \cD(e^{i\del\pi(x)})$, for all $x\in \Omega_\p$. Thus, the assertion follows from \cite[Prop. 3.5]{FNO23}, which is stated for unitary representations but its proof can be seen to hold for Hilbert representations as well. Here the main point is to show that the map
$$G\times i\Omega_\p\times \wt{K}_\C\to \cH,\quad (g,ix,k)\mapsto \pi(g)e^{i\del\pi(x)}\pi^{v}(k)
$$
factors through a holomorphic map $\wxi\to \cH$ extending $\pi^v:G\to \cH$, which works analogously since the proof does not use the unitarity.
\end{proof}
\begin{cor}\label{corKrSt}
Let $(\pi,F)$ be a smooth Fr\'{e}chet globalization of moderate growth of a Harish-Chandra module. If $v\in F^{[K]}$ is a $K$-finite vector, then the orbit map $\pi^v:G\to F$ extends to a holomorphic map $\pi^v:\wxi\to F$.
\end{cor}
\begin{proof}
Let $(\scp_n)_{n\in \N}$ be a family of $K$-invariant positive-semidefinite sesquilinear forms such that the corresponding Hilbert seminorms $\|\cdot\|_n$ are $G$-continuous and induce the topology on $F$ (cf. Theorem \ref{ThmHilbSemi}). Let 
$$F_n:=\{v\in F\;|\; \|v\|_n=0\},\quad \cH_n:=\overline{F/F_n}^{\|\cdot \|_n}\qand \cH:=\prod_{n\in \N}\cH.
$$
Equip $\cH:=\prod_{n\in \N}\cH_n$ with the product topology and recall the embedding of Fr\'{e}chet representations $\Phi:F\hookto \cH$ from Remark \ref{remFrechGlob}. Let $v\in F^{[K]}$ be a $K$-finite vector. Then, $v_n:=v\;\text{mod}\;F_n\in \cH_n$ is $K$-finite and the restriction $$\pr_n\circ\Phi\vert_{F^{[K]}}:F^{[K]}\to \cH_n^{[K]}$$ is a $(\g,K)$-module homomorphism. In particular, Theorem \ref{thmKrSt} implies that the orbit map of $v_n\in \cH_n$ extends to a holomorphic map 
$$
\pi_n^{v_n}:\wxi\to \overline{\pi_n(G)v_n}\subseteq \cH_n\quad\text{and thus}\quad\pi^v:=\Phi^{-1}\circ \prod_{n\in \N}\pi_n^{v_n}:\wxi \to F
$$
is a holomorphic extension of the orbit map $\pi^v:G\to F,\;g\mapsto\pi(g)v$.
\end{proof}
\newpage
\section{Growth of the complexified Iwasawa component maps}\label{secIwasawa}
The Kr\"otz--Stanton Extension Theorem provides a holomorphic extension of the orbit maps of $K$-finite vectors for Hilbert globalizations of Harish--Chandra modules to the principal crown bundle. In this section, we develop tools with which we can estimate the growth as one approaches the boundary of the principal crown bundle. Here the main point is that $\xig=G_\R \exp(i\Omega_\p)K_\C$ is contained in $\NAKC\subseteq G_\C$ and to show that one can control the growth of the complexified Iwasawa component maps in terms of scale functions. In Subsection \ref{subsAlgGeo}, we discuss the structure of $\NAKC$ as a smooth affine variety, which allows us to relate the growth of the $K_\C$- and $N_\C$-component to the growth of the $A_\C$-component. This is the main result of Subsection \ref{subsecAC}. In Subsection \ref{subsecasympA}, we then prove our main result on the asymptotic behaviour of the Iwasawa component maps as one approaches the boundary of the principal crown domain.
\subsection{The complexified Iwasawa domain}\label{subsAlgGeo}
Let $G$ be a simply connected semisimple Lie group with Iwasawa decomposition $G=KAN$ and universal complexification $\eta_G:G\to G_\C$. For the integral subgroups $K_\C,A_\C,N_\C\subseteq G_\C$ with respective Lie algebras $\k_\C,\a_\C,\n_\C\subseteq \g_\C$, one considers the \textit{complexified Iwasawa domain}
$$\KANC\subseteq G_\C.
$$
\begin{ex}\label{exsl2}
\rm{Let $G=\SL(2,\R)$ and $G_\C=\SL(2,\C)$. Then the Iwasawa decomposition of $G$ is given by
\begin{equation}\label{EqIwasawaSL2}
g=\begin{pmatrix}
a&b\\
c&d
\end{pmatrix}=\frac{1}{\sqrt{a^2+c^2}}\begin{pmatrix}
a&-c\\
c&a
\end{pmatrix}\begin{pmatrix}
\sqrt{a^2+c^2}&0\\
0& \frac{1}{\sqrt{a^2+c^2}}
\end{pmatrix}\begin{pmatrix}
1&\frac{ab+cd}{a^2+c^2}\\
0& 1
\end{pmatrix}.
\end{equation}
The Iwasawa component maps have a multivalued holomorphic extension to the subdomain $\{a^2+c^2\neq 0\}$ of $\SL(2,\C)$ and the asymptotic at the boundary can be described in terms the function $\Delta:G_\C \to \C$, $g\mapsto a^2+c^2$. Another aspect of this example, which turns out to be a general phenomenon, is that the entries of the $N_\C$-component are polynomials in the entries of elements of $G_\C$ divided by $\Delta$. 
}
\end{ex}
\begin{defin}
\rm{We recall some concepts from Algebraic Geometric. (cf. \cite{Hu75}).
\begin{itemize}
\item[(i)] A subset $X\subseteq \C^n$ is called \textit{Zariski closed} or an \textit{affine varieties} if there exist polynomials $f_1,\dots, f_k\in \C[z_1,\dots,z_n]$ such that
$$X=A(f_1,\dots,f_k):=\{z\in \C^n \;|\;f_i(z)=0,\qforall 1\leq i\leq k\}.
$$ 
The subsets $\C^n \setminus A(f)$ are called \textit{principal open subsets}, for $f\in \C[z]$.
\item[(ii)] Let $A_1\subseteq \C^n$ and $A_2\subseteq \C^m$ be affine varieties. A \textit{morphism} $f:A_1\to A_2$ is a restriction of a polynomial map $\C^n\to \C^m$ mapping $A_1$ to $A_2$. 
\item[(iii)] The coordinate ring $\C[X]$ of $X=A(f_1,\dots,f_k)$ is defined as $$\C[z_1,\dots,z_n]/(f_1,\dots,f_k),$$
where $(f_1,\dots,f_k)\tri \C[z_1,\dots,z_n]$ is the ideal generated by $f_1,\dots,f_k$. There is a one-to-one correspondence between the set of morphisms $f:X\to \C$ and $\C[X]$.
\end{itemize}
}
\end{defin}
\begin{rem}
\rm{Let $X=\C^n\setminus A(f)$ be a principal open subset. Then, the map
$$X\to A(1-z_0f)=\{(z_0,z)\in \C\times \C^{n}\;|\;1-z_0f(z)=0\},\quad z\mapsto \left(\tfrac{1}{f(z)},z\right)
$$
identifies $X$ with an affine variety in $\C^{n+1}$. The coordinate ring of $X$ is given by
$$\C[X]=\left\{\frac{p}{f^k}\;|\;p\in \C[z_1,\dots,z_n],\quad k\in \N_0\right\}.
$$
}
\end{rem}
We return to the general setting and describe the principal open subset $\KANC\subseteq G_\C$. Let $r:=\dim_\R\a$ be the real rank of $\g$ and let $(\pi_i,V_i)$ be the fundamental finite-dimensional irreducible spherical representations of $G$, for $i=1,\dots,r$, whose existence is guaranteed by the Cartan-Helgason Theorem (cf. \cite[Thm. V.4.1]{He84}). Fix a scalar product $\scp_i$ of $V_i$ such that
$$\langle v,\pi(g)w\rangle_i=\langle \pi(\theta(g))^{-1}v,w\rangle_i,\qfor v,w\in V\qand g\in G.
$$
Furthermore, consider highest weight vectors $v_i\in V_i$ with highest weight $\omega_i\in \a^\ast$ and $K$-fixed vectors $w_i\in V_i^K$ such that $\langle v_i,w_i\rangle_i=1$. Since $G_\C$ is simply connected, there exists an involutive holomorphic automorphism $\theta_\C:G_\C\to G_\C$ extending the Cartan involution and $G_\C^{\theta_\C}$ is connected. Therefore, $G_\C^{\theta_\C}=K_\C$ and \cite[Thm. 2.5]{Ol87} applied to the symmetric pair $(G_\C,K_\C)$ implies that
\begin{equation}\label{KANC}
\KANC=\left\{g\in G_\C\;|\;\Delta(g)\neq 0\right\},\qfor \Delta(g):=\prod_{i=1}^r \langle \pi_i(g)v_i,w_i\rangle_i.
\end{equation}
Since $\Delta:G_\C\to \C$ is a product of matrix coefficients, it is contained in the coordinate ring $\C[G_\C]$ of the linear algebraic group $G_\C$ and $\KANC=\Delta^{-1}(\C^\times)$ is an open Zariski dense affine subvariety of $G_\C$. Furthermore, the coordinate ring of the principal open subset $\KANC\subseteq G_\C$ is given by
$$\C[\KANC]=\left\{\frac{p}{\Delta^k}\;|\;p\in \C[G_\C],\quad k\in \N_0 \right\}.
$$
Let $\mu:=\sum_{i=1}^r \omega_i\in \a^\ast$. Then, for $k\in K_\C,a\in A_\C$ and $n\in N_\C$, one has
\begin{equation}\label{Delta}
\Delta(kan)=\prod_{i=1}^r \langle \pi_i(kan)v_i,w_i\rangle_i=a^\mu\in \C^\times,\qfor \exp(H)^\mu:=e^{\mu(H)}.
\end{equation}
In the proof of \cite[Thm. 1.8]{KS04} it was shown that, for $F:=K_\C \cap A_\C$,
\begin{equation}\label{EqKANC}
(K_\C\times_F A_\C)\times N_\C\to \KANC,\quad ([(k,a)],n)\mapsto kan
\end{equation}
is a biholomorphism. Hence, there exists an $N_\C$-projection map $\eta:\KANC\to N_\C$, which extends the $N$-projection map $\eta:G=KAN\to N$. 
\begin{lem}
The $N_\C$-projection $\eta:\KANC\to N_\C$ is a morphism of affine varieties.
\end{lem}
\begin{proof}
The symmetric space $K_\C \backslash G_\C$ is a smooth affine variety which can be realized as a Zariski closed subset of $G_\C$ by the Cartan embedding
$$K_\C\backslash G_\C\to G_\C^{-\theta_\C}:=\{g\in G_\C\;|\;\theta(g)^{-1}=g\},\quad K_\C g \mapsto \theta(g)^{-1}g.
$$
The corresponding algebraic $G_\C$-right action on $G_\C^{-\theta_\C}$ is given by
$$G_\C^{-\theta_\C}\times G_\C\to G_\C^{-\theta_\C},\quad (g_0,g)\mapsto \theta(g)^{-1}g_0g
$$
and the $A_\C N_\C$-orbit of the identity $e$ given by
$$\cS:=\{\theta(n)^{-1}a^2n\;|\;a\in A_\C,\; n\in N_\C\}=\{\theta(g)^{-1}g\;|\;g\in \KANC\}.$$
The stabilizer of $e$ in $A_\C N_\C$ is given by $F=K_\C\cap A_\C=\{a\in A_\C \;|\;a^2=e\}$ (cf. proof of \cite[Prop. 1.3]{KS04}) and
$$\Phi:A_\C\times N_\C \to \cS,\quad (a,n)\mapsto \theta(n)^{-1}an
$$
is a diffeomorphism. Additionally, $\Phi:A_\C\times N_\C\to G_\C^{-\theta_\C}$ is a morphism of affine varieties. In particular, $\cS$ is a Zariski open subset of $G_\C^{-\theta_\C}$ and as such a smooth affine variety. Since $A_\C\times N_\C$ is an irreducible affine variety and $\cS$ is normal, it follows from an application of Zariski's Main Theorem (cf. \cite[Prop. 3.2]{BFMQ22}) that $\Phi$ is an isomorphism. Therefore,
$$\eta(g)=\left(\pr_{N_\C}\circ \Phi^{-1}\right)(\theta(g)^{-1}g),\qfor g\in \KANC,
$$
implies that $\eta$ is a morphism.
\end{proof}
\begin{cor}\label{corUnip}
Let $(\sigma,W)$ be a nilpotent representation of $N_\C$, i.e. there exists $k\in \N_0$ such that $d\sigma(\n_\C)^k=\{0\}$. For $v,w\in W$ and $\sigma^{v,w}:N_\C\to \C,\;n\mapsto \langle v,\sigma(n)w\rangle$, one has $$\sigma^{v,w}\circ \eta\in \C[\KANC].$$
\end{cor}
\begin{proof}
As $\sigma$ is nilpotent, $\sigma^{v,w}$ is a morphism, which proves the assertion.
\end{proof}
\begin{rem}
\rm{Note that the nilpotency of the representation $(\sigma,W)$ is necessary since, for $G=\SL(2,\R)$, the characters
$$\chi_v:N_\C\to \C^\times,\quad \begin{pmatrix}
1&w\\
0&1
\end{pmatrix}\mapsto e^{vw},\qfor v\in \C^\times,
$$
are not morphisms of affine varieties. For our purposes, nilpotent representations are sufficient, since restrictions of finite-dimensional representations of $G_\C$ to $N_\C$ are nilpotent.
}
\end{rem}
\subsection{Scaling properties of the Iwasawa component maps}\label{subsecAC}
In this subsection, we recall some basic properties of scale functions and discuss the growth of the Iwasawa component maps in terms of these scales as one approaches the boundary of the domain $\KANC$. 
\begin{defin}
\rm{For a topological group $G$, a function $s:G\to [1,\infty)$ is called a \textit{scale function} if
\begin{itemize}
\item[(i)] $s$ and $s^{-1}$ are locally bounded,
\item[(ii)] $s$ is submultiplicative, i.e. $s(gh)\leq s(g)s(h)$, for all $g,h\in G$.
\end{itemize}
One defines a partial order on the set of scale functions as follows. For two scale functions $s,s^\prime:G\to \R^+$, consider
$$s\preccurlyeq s^\prime \quad :\Leftrightarrow\quad (\exists C>0,N\in \N)(\forall g\in G),\quad s(g)\leq C\cdot s^\prime(g)^N.
$$
An equivalence class $[s]$ of scale functions is called a \textit{scale structure} on $G$.
}
\end{defin} 
\begin{defin}\label{defScale}
\rm{(cf. \cite{BK14}) Let $G$ be a connected reductive Lie group with Cartan decomposition $G=K\exp(\p)$ and let $\|\cdot \|$ be a $K$-invariant norm on $\p$. Then, one defines the \textit{maximal scale function}
$$s_{\rm max}^G:G\to [1,\infty),\quad k\exp(X)\mapsto e^{\|X\|},\qfor k\in K,X\in \p.
$$
}
\end{defin}
\begin{rem}
\rm{Let $G$ be a connected reductive Lie group and $s:G\to [1,\infty)$ be a scale function. Then, \cite[Lemme 2]{Ga60} implies that $s\preccurlyeq s_{\rm max}^G$, which justifies the name maximal scale structure. Note that the maximal scale function depends on the choice of norm but its scale structure does not. For $\u:=\k\oplus i\p$ and $U:=\langle \exp(\u)\rangle\subseteq G_\C$, consider the $U$-invariant norm
\begin{equation}\label{defRho}
\|\cdot\|:i\u\to [0,\infty),\quad x\mapsto \rho(x):=r_{\rm spec}(\ad(x)).
\end{equation} 
Then, the restriction of $\|\cdot\|$ to $\p$ is $K$-invariant, and for our purposes, this choice of norm is the most convenient.
}
\end{rem}
\begin{lem}
Let $G$ be a connected semisimple Lie group with universal complexification $\eta_G:G\to G_\C$. Then, $s_{\rm max}^{G_\C}\vert_{K_\C}=s_{\rm max}^{K_\C}$.
\end{lem}
\begin{proof}
A Cartan decomposition of $\g_\C$ is given by $\g_\C=\u\oplus i\u$ and thus $$G_\C=U\exp(i\u),\qfor U:=\langle \exp(u)\rangle,$$
is a global Cartan decomposition of $G_\C$. A global Cartan decomposition of the reductive Lie group $K_\C$ is given by $K_\C=K_\R\exp(i\k)$. In particular, for $k\in K_\R$ and $iX\in i\k\subseteq i\u$, one has
\begin{align*}
s_{\rm max}^{G_\C}(k\exp(iX))=e^{\|iX\|}=s_{\rm max}^{K_\C}(k\exp(iX)).\tag*{\qedhere}
\end{align*}
\end{proof}
The following lemma can be found in \cite[Lemma 2.1]{BK14} for linear reductive groups.
\begin{lem}\label{lemScaleFinRep}
Let $G$ be a connected semisimple Lie group and let $\pi:G\to \GL(V)$ be a finite-dimensional complex representation with discrete kernel. Then,
$$s_\pi:G\to [1,\infty),\quad g\mapsto \|\pi(g)\|+\|\pi(g)^{-1}\|
$$
is a scale function on $G$ and $[s_\pi]=[s_{\rm max}^G]$.
\end{lem}
\begin{proof}
That $s_\pi$ is a scale function is trivial. The representation $\pi$ extends to a holomorphic representation of $G_\C$. Fix a $U$-invariant scalar product on $V$. Then, $s_\pi$ extends to $G_\C$ and one has
$$s_\pi(u\exp(X)):=\|\pi(u\exp(X))\|+\|\pi(u\exp(X))^{-1}\|\geq e^{r_{\rm spec}(d\pi(X))},
$$
for $u\in U$ and $X\in i\u$, where $r_{\rm spec}$ denotes the spectral radius. Since $r_{\rm spec}\circ d\pi$ defines a $U$-invariant norm on $i\u$, the assertion follows immediately.
\end{proof}
\begin{rem}
\rm{Recall from (\ref{EqKANC}) that $\KANC\cong (K_\C\times_F A_\C)\times N_\C$ as complex manifolds, for $F=K_\C\cap A_\C$. In particular, there exist holomorphic projections
$$\KANC\to K_\C/F,\quad \KANC \to A_\C /F\qand \KANC \to N_C
$$
and thus one can write every element $g\in \KANC$ as $g=\kappa(g)\alpha(g)\eta(g)$, where $$\kappa(g)F\in K_\C/F,\quad \alpha(g)F\in A_\C/F\qand \eta(g)\in N_\C$$
are uniquely determined. Since $F=K_\C\cap A_\C\subseteq \exp(i\a)\subseteq U$ and $s_{\rm max}^{G_\C}$ is $U$-bi-invariant, it follows that
$$s_{\rm max}^{G_\C}(\kappa(g)),\quad s_{\rm max}^{G_\C}(\alpha(g))\qand s_{\rm max}^{G_\C}(\eta(g)),
\qfor g\in \KANC,$$
are independent of the representants $\kappa(g)$ and $\alpha(g)$. We can therefore speak of scaling properties of the Iwasawa component maps without having to worry about the disambiguities of the decomposition $g=\kappa(g)\alpha(g)\eta(g)$.
}
\end{rem} 
\begin{prop}\label{propScales}
Let $G$ be a connected semisimple Lie group.
\begin{itemize}
\item[{\rm (i)}] For $H_1,H_2\in \a$, one has $s_{\rm max}^{G_\C}(\exp(H_1+iH_2))=e^{r_{\rm spec}(\ad(H_1))}$. In particular, if $\nu \in \a^\ast$, then there exists $N>0$ such that
$$|a^\nu|\leq s_{\rm max}^{G_\C}(a)^N,\qforall a\in A_\C.
$$
\item[{\rm~(ii)}] There exist $C>0$ and $M,N\geq 0$ such that{\rm~
$$s_{\rm max}^{G_\C}(\eta(g))\leq C\cdot s_{\rm max}^{G_\C}(g)^M \cdot s_{\rm max}^{G_\C}(\alpha(g))^N, \quad for\; all\quad  g\in K_\C A_\C N_\C.
$$}
\item[{\rm~(iii)}] There exist $C>0$ and $M,N\geq 0$ such that\rm{
$$s_{\rm max}^{K_\C}(\kappa(g))\leq C\cdot s_{\rm max}^{G_\C}(g)^M \cdot s_{\rm max}^{G_\C}(\alpha(g))^N, \quad for\; all\quad  g\in K_\C A_\C N_\C.
$$}
\end{itemize}
\end{prop}
\begin{proof}
(i) follows immediately from the definition of the maximal scale function and
$$|\exp(H_1+iH_2)^\nu|=e^{\nu(H_1)}\leq e^{\|\nu\|\cdot r_{\rm spec}(\ad(H_1))}=s_{\rm max}^{G_\C}(\exp(H_1+iH_2))^{\|\nu\|},$$
for $\|\nu\|:=\sup\{|\nu(H)|\;|\;H\in \a,\; r_{\rm spec}(\ad(H))\leq 1\}$. \\

(ii) In view of Lemma \ref{lemScaleFinRep}, it suffices to prove the assertion for the scale function $s_\pi$, for a finite-dimensional representation $(\pi,V)$ of $G_\C$ with discrete kernel. Fix a $U$-invariant scalar product on $V$ and let $\{v_1,\dots,v_n\}\subseteq V$ be a basis. Then, $\pi(g)^\ast=\pi(\theta(g))^{-1}$ implies that
$$s_\pi(g)\leq\sum_{i,j=1}^n |\langle v_i,\pi(g)v_j\rangle|+|\langle \pi(\theta(g))v_i,v_j\rangle|.
$$
Since $A_\C N_\C$ is a solvable Lie group with unipotent radical $N_\C$, the restrictions of the representations $\pi$ and $\pi\circ \theta$ to $N_\C$ are nilpotent (cf. \cite[Cor. 5.4.11]{HN12}). In particular, Corollary \ref{corUnip} implies that
$$\pi^{v_i,v_j}\circ \eta:\KANC \to \C\qand (\pi\circ \theta)^{v_i,v_j}\circ \eta:\KANC \to \C
$$
are morphisms, i.e. they are contained in the coordinate ring
$$\C[\KANC]=\left\{\frac{p}{\Delta^k}\;|\;p\in \C[G_\C],\quad k\in \N_0 \right\},
$$
where $\Delta:G_\C\to \C$ is as in (\ref{KANC}). The coordinate ring $\C[G_\C]$ of $G_\C$ is generated by matrix coefficients of finite-dimensional representations of $G_\C$, which are dominated by multiples of powers of $s_{\rm max}^{G_\C}$. Hence, there exist $C,M,N>0$ such that
$$s_\pi(\eta(g))\leq C\cdot s_{\rm max}^{G_\C}(g)^M\cdot \sum_{k=0}^N |\Delta^{-k}(g)|
$$
and since $\Delta(g)=\Delta(\alpha(g))=\alpha(g)^\mu$, for $\mu=\sum_{i=1}^r\omega_i\in \a^\ast$, the assertion follows from (i).\\

(iii) The submultiplicativity of the scale function $s_{\rm max}^{G_\C}$ and 
$$\kappa(g)=g\eta(g)^{-1}\alpha(g)^{-1},\qfor g=\kappa(g)\alpha(g)\eta(g)\in K_\C A_\C N_\C,$$
imply that
$$s_{\rm max}^{G_\C}(\kappa(g))\leq s_{\rm max}^{G_\C}(g)s_{\rm max}^{G_\C}(\alpha(g)^{-1})s_{\rm max}^{G_\C}(\eta(g)^{-1}).
$$
In view of $s_{\rm max}^{G_\C}(g_0^{-1})=s_{\rm max}^{G_\C}(g_0)$, for all $g_0\in G_\C$, (iii) follows from (ii).
\end{proof}  
\subsection{Asymptotic behaviour at the boundary of the crown}\label{subsecasympA}
It was shown in multiple papers \cite{KS04,Hu02,GM03,Ba03} that one has the inclusion 
\begin{equation}\label{EqIncl}
\Xi_{G_\C}=G_\R\exp(i\Omega_\p)K_\C\subseteq N_\C A_\C K_\C,\qfor \Omega_\p=\left\{x\in \p\;|\;r_{\rm spec}(\ad(x))<\pitwo\right\}.
\end{equation}
To our knowledge, the shortest proof of this fact is \cite[Cor. 3.3]{KSch09}. Let 
$$q_{\xig}:\wxi\to \xig\qand q_{K_\C}:\wt{K}_\C\to K_\C$$
denote the universal coverings of $\xig$ and $K_\C$ respectively. Then, with the same arguments as in \cite[Thm. 1.8]{KS04} and lifting everything to the universal covering spaces, it follows that there exist \textit{complexified Iwasawa component maps}
\begin{align}
&\wxi\times K\to \wt{K}_\C, \quad(x,k)\mapsto \kappa(x^{-1}k), \label{Iwasawa1}\\
&\wxi\times K\to \a_\C, \;\quad(x,k)\mapsto H(x^{-1}k), \label{Iwasawa2}\\
&\wxi\times K\to N_\C, \quad(x,k)\mapsto \eta(x^{-1}k), \label{Iwasawa3}
\end{align}
which are holomorphic in the first variable and analytic  in the second variable and satisfy
$$q_{\xig}(x)^{-1}k=q_{K_\C}(\kappa(x^{-1}k))\exp(H(x^{-1}k))\eta(x^{-1}k),\qforall x\in \wxi\qand k\in K.
$$
We prefer the abuse of notation of writing $x^{-1}k$ in the arguments of the Iwasawa component maps, for $x\in \wxi$ and $k\in K$, as the passing to the coverings does not contribute to the asymptotic and it makes the notation clearer in our opinion. 
\begin{rem}\label{remHolExtKfin}
\rm{We now briefly recall the motivation for the results of this subsection. Let $(\sigma,W)$ be a finite-dimensional representation of $P_{\rm min}$ which is unitarizable when restricted to $M$. Let $(\pi_\sigma,\cH_\sigma)$ be the corresponding principal series representation and $f\in \cH_\sigma^{[K]}$ be a $K$-finite vector (cf. (\ref{defPrinc1})). Since $f:K\to W$ is a left $K$-finite smooth function, $$\cE:=\spann_\C \{L_k f\;|\;k\in K \} \subseteq C^\infty(K,W),\qfor L_{k_0}f(k):=f(k_0^{-1}k),$$ is a finite-dimensional $K$-invariant subspace and thus the orbit map $k\mapsto L_{k^{-1}} f$ extends to a holomorphic map $\wt{K}_\C\to \cE\subseteq C^\infty(K,W)$. Post-composing with the evaluation at $e$ defines a holomorphic extension of $f$ to $f:\wt{K}_\C\to W$. The holomorphic extension of the orbit map $$\pi_\sigma^f:G\to \cH_\sigma,\quad g\mapsto \pi_\sigma(g)f$$
is explicitly given by 
$$(e^{it\del\pi_\sigma(h)}f)(k)=\sigma_\rho(\beta(\exp(-ith)k))^{-1}f(\kappa(\exp(-ith)k)),$$
for all $h\in \del\Omega_\p$ and $|t|<1$.}
\end{rem}
As we are interested in asymptotic estimates $\|e^{it\del\pi_\sigma(h)}\|$ as $t\nearrow 1$ and the results from the previous subsection reduce the $K_\C$- and $N_\C$-asymptotic to the $A_\C$-asymptotic, we study the asymptotic behaviour of
\begin{equation}\label{EqAsymp}
s_{\rm max}^{G_\C}(\alpha(\exp(-ith)k)),\qfor h\in \del\Omega_\p,\quad k\in K\quad\text{as}\quad t\nearrow 1.
\end{equation}
We may w.l.o.g. assume that $G$ is \textbf{linear} and identify $G$ with $G_\R=\eta_G(G)$. In particular, the subgroup $K\subseteq G$ is compact. Let $\Sigma:=\Sigma(\g,\a)\subseteq\a^\ast$ be the set of restricted roots associated with $\a$ and $\Omega_\a:=\a\cap \Omega_\p$. For $h\in \a$, one has
$$\spec(\ad(h))=\{\gamma(h)\;|\;\gamma\in \Sigma\}\qandthus \Omega_\a=\{h\in \a\;|\;|\gamma(h)|<\tpitwo,\;\text{for all}\; \gamma\in \Sigma\}.$$
Let $\{\omega_1,\dots,\omega_r\}\subseteq \a^\ast$ be the fundamental spherical weights one obtains from the Cartan-Helgason Theorem (cf. \cite[Thm. V.4.1]{He84}). Then, $\{\omega_1,\dots,\omega_r\}$ is a basis of $\a^\ast$ and thus
$$\|\cdot\|:\a\to [0,\infty),\quad H\mapsto \sum_{i=1}^r|\omega_i(H)|
$$
defines a norm and, for $c>0$ sufficiently large, one has
$$s_{\rm max}^{G_\C}(\exp(H_1+iH_2))\leq e^{c \|H_1\|},\qforall H_1,H_2\in \a.$$
\begin{rem}
\rm{Let $(\pi,V)$ be a finite-dimensional complex representation of $G_\C$. Let $\Theta$ be the antiholomorphic extension of $\theta$ to $G_\C$. Then, $\Theta$ is the Cartan involution on $G_\C$ with respect to $U=\langle \exp(\k+i\p)\rangle$ and we choose a $U$-invariant inner product $\scp$ on $V$. Then $d\pi(\a)\subseteq \gl(V)$ is an abelian subalgebra of hermitian endomorphisms, i.e. $V$ decomposes into $\a$-weight spaces. Assume that $V$ is a highest weight module, for the restricted root system, with highest weight $\lambda\in \a^\ast$, i.e. $V$ is generated by a weight space $V_\lambda$ which is fixed pointwise by $N_\C$. For $\mu\in \a^\ast$, we denote the corresponding character by $$\mu:A\to (0,\infty),\quad a=\exp(H)\mapsto a^\mu:=e^{\mu(H)},\qfor H\in \a.$$
We denote the $\a$-weights of $V$ with $\cP\subseteq \a^\ast$. One has a weight space decomposition
\begin{align*}
V=\sum_{\mu\in \cP} V_\mu,\qfor V_\mu:=\{v\in V\;|\; \pi(a)v=a^\mu v,\qforall a\in A\}.
\end{align*}
Since the elements in $d\pi(\a)$ are hermitian, the weight spaces are mutually orthogonal. For a normalised highest weight vector $v_\lambda\in V_\lambda$, it follows from the fact that $\pi\vert_K$ is unitary that 
\begin{align*}
\|\pi(kan)v_\lambda\|^2=a^{2\lambda},\qfor k\in K,\;a\in A,\;n\in N.
\end{align*}
For $k\in K$, let $v_\mu(k)\in V_\mu$ be the unique element such that $\pi(k)v_\lambda=\sum_{\mu\in \cP}v_\mu(k)$ and observe that 
$$\alpha(ak)^{2\lambda}=\|\pi(ak)v_\lambda\|^2=\sum_{\mu\in \cP}a^{2\mu}\|v_\mu(k)\|^2,$$
for $a\in A$ and $k\in K$, since the weight spaces are orthogonal.
}
\end{rem}
In the following, we consider a fixed complex finite-dimensional highest restricted weight representation $(\pi,V)$ of $G_\C$ with highest weight $\lambda\in \a^\ast$ and $\a$-weights $\cP\subseteq \a^\ast$.
\begin{prop}\label{formulaalpha}
In the notation of the previous remark and, for  $h\in\del\Omega_\a$ and $|t|<1$, one has  
\begin{equation}\label{formalpha}
\left|\alpha(\exp(-ith)k)^{2\lambda} \right|^2=\sum_{\mu,\nu\in \cP}\cos(2t(\mu-\nu)(h))\|v_\mu(k)\|^2\|v_\nu(k)\|^2.
\end{equation}
In particular, one has $|\alpha(\exp(-ith)k)^{\lambda}|\leq 1$, for $|t|<1$.
\end{prop}
\begin{proof}
Recall from (\ref{Iwasawa2}) that the Iwasawa component map $H:G\to \a$ has a holomorphic extension $H:\wt{\Xi}_{G_\C}\to \a_\C$. Therefore, the map $\R\ni t\mapsto \alpha(\exp(-th)k)\in A$ extends to a holomorphic map on the open strip $\cS_{\pm1}:=\{z\in \C\;|\; |\text{Im}(z)|<1\}\to \a_\C$ and thus one has by uniqueness of holomorphic extension
\begin{align*}
0\neq e^{2\lambda(H(\exp(-ith)k))}= \alpha(\exp(-ith)k)^{2\lambda}= \sum_{\mu\in \cP}e^{-2it\mu(h)}\|v_\mu(k)\|^2,
\end{align*}
for all $|t|<1$ and $k\in K$. This in particular implies that 
\begin{align*}
\left|\alpha(\exp(-ith)k)^{2\lambda} \right|^2=&
\sum_{\mu,\nu\in \cP}e^{2it(\nu-\mu)(h)}\|v_\mu(k)\|^2\|v_\nu(k)\|^2\\
=&\sum_{\mu,\nu\in \cP}\cos(2t(\mu-\nu)(h))\|v_\mu(k)\|^2\|v_\nu(k)\|^2.\tag*{\qedhere}
\end{align*}
\end{proof}
\begin{lem}\label{coeff}
For $h\in \del\Omega_\a$, $k\in K$ and $0\leq t<1$, one has
\begin{equation}\label{fhk}
f_{h,k}(t):=\left|\alpha(\exp(-ith)k)^{2\lambda} \right|^2=\sum_{n\geq 0} a_n(h,k) \left(1-t \right)^{n},
\end{equation}
where the coefficient functions $a_n:\del\Omega_\a\times K\to \R$ are continuous and
$$
\|a_n\|_\infty\leq \frac{(2C_{\rm max})^{n}}{n!},\qfor C_{\rm max}:=\max\{|(\mu-\nu)(h)|\;:\;\mu,\nu \in \cP\}.$$
\end{lem}
\begin{proof}
Using Taylor expansion of (\ref{formalpha}) in $t=1$ and absorbing the alternating factors $(-1)^n$, one obtains from the derivatives of $\sin$ and $\cos$, by changing $(t-1)^n$ to $(1-t)^n$ in (\ref{fhk}), that the coefficient functions $a_n$ are explicitly given by
\begin{align*}
a_{2n}(h,k)&=\frac{2^{2n}}{(2n)!} \sum_{\mu,\nu\in\cP}\cos(2(\mu-\nu)(h))(\mu-\nu)(h)^{2n}\|v_\mu(k)\|^2\|v_\nu(k)\|^2\\
 a_{2n+1}(h,k)&= \frac{2^{2n+1}}{(2n+1)!} \sum_{\mu,\nu\in\cP}\sin(2(\mu-\nu)(h))(\mu-\nu)(h)^{2n+1}\|v_\mu(k)\|^2\|v_\nu(k)\|^2.
\end{align*}
In particular, $\sum_{\mu,\nu\in \cP}\|v_\mu(k)\|^2\|v_\nu(k)\|^2=1$ implies $\|a_n\|_\infty\leq \frac{1}{n!}(2C_\text{max})^{n}$.
\end{proof}
\begin{lem}\label{lemN0}
For $N\in \N_0$, define 
\begin{align*}
A_N:=\{(h,k)\in \del\Omega_\a\times K\;|\;a_{N}(h,k)\neq 0,\; a_n(h,k)=0\qforall 0\leq n<N\}.
\end{align*}
The set $A_N$ coincides with the set of $(h,k)\in \del\Omega_\a\times K$ satisfying
\begin{align*}
0<\lim_{t\to 1}(1-t)^{N}|\alpha(\exp(-ith)k)^{-2\lambda}|^2<\infty
\end{align*}
and there exists $N\in \N_0$ maximal such that $A_N\neq \emptyset$.
\end{lem}
\begin{proof}
It follows from $\|a_n\|_\infty\leq \frac{1}{n!}(2C_{\rm max})^n$ that $$(-1,1)\times \del\Omega_\a\times K\to (0,\infty),\quad (t,h,k)\mapsto \sum_{n\geq 0}a_n(h,k)(1-t)^n$$
is continuous. In particular, one can interchange the limit as $t\to 1$ and the summation which together with
\begin{align*}
|\alpha(\exp(-ith)k)^{2\lambda}|^2=\sum_{n\geq 0}a_n(h,k)(1-t)^n
\end{align*}
implies that $A_N$ coincides  with the set of $(h,k)\in \del\Omega_\a\times K$ satisfying 
\begin{align*}
0<\lim_{t\to1}(1-t)^{N}|\alpha(\exp(-ith)k)^{-2\lambda}|^2<\infty.
\end{align*}
There exists $N\in \N_0$ such that $A_N\neq \emptyset$, since otherwise $a_n(h,k)=0$ would follow, for all $n\in \N_0$ and $(h,k)\in \del\Omega_\a\times K$, which contradicts $$\left|\alpha(\exp(-ith)k)^{2\lambda} \right|^2\neq 0.$$ Hence such an $N\in \N_0$ exists and there exists a maximal one since otherwise there exist infinitely many $n\in \N_0$ and $(h_n,k_n)\in\del\Omega_\a\times K$ such that $a_m(h_n,k_n)=0$ for all $0\leq m < n$. We may w.l.o.g. assume that $(h_n,k_n)\to (h_0,k_0)\in \del\Omega_\a\times K$ as $\del\Omega_\a\times K$ is compact and thus
\begin{align*}
&\lim_{n\to \infty} \sum_{m\geq 0} a_m(h_n,k_n)\left(1-t\right)^{m}=\lim_{n\to \infty} \sum_{m\geq n} a_m(h_n,k_n)\left(1-t\right)^{m}\\
\leq& \lim_{n\to \infty} \sum_{m\geq n} \frac{(2C_\text{max})^{m}}{m!}\left(1-t\right)^{m}=0
\end{align*}
contradicts $\left|\alpha(\exp(-ith_0)k_0)^{2\lambda} \right|^2\neq 0$.
\end{proof}
\begin{prop}
Let $X$ be a compact topological space and suppose there are continuous functions $a_n:X\to \C$ such that $\|a_n\|_\infty\leq \frac{C^n}{n!}$, for some $C>0$, and let
$$f_x(r):=\sum_{n=0}^\infty a_n(x) (1-r)^n\neq 0,\qforall x\in X\qand r\in (0,1).
$$
Suppose there exists $N\in \N_0$ maximal such that
$$A_N:=\{x\in X\;|\; a_n(x)=0,\qforall 0\leq n\leq N-1,\quad a_N(x)\neq 0\}
$$
is non-empty. Then, one has
$$\limsup_{r\nearrow 1}\left( \sup_{x\in X} \frac{(1-r)^N}{|f_x(r)|}\right)<\infty.
$$
\end{prop}
\begin{proof}
We argue by contradiction and thus suppose that there exists a sequence $r_n$ in $(0,1)$ such that $r_n\nearrow 1$ as $n\to \infty$ and
$$\lim_{n\to \infty}  \sup_{x\in X} \frac{(1-r_n)^N}{|f_x(r_n)|}=\infty.
$$
Since $X$ is compact, there exists a sequence $x_n\in X$ such that 
\begin{align*}
\sup_{x\in X} \frac{1}{|f_x(r_n)|}=\frac{1}{|f_{x_n}(r_n)|}.
\end{align*}
In particular, it follows from (\ref{fhk}) that
\begin{equation}\label{limit2}
0=\lim_{n\to \infty}(1-r_n)^{-N}|f_{x_n}(r_n)|=\sum_{l=0}^\infty a_l(x_n)(1-r_n)^{l-N}.
\end{equation}
Since $\|a_l\|_\infty \leq \frac{C^l}{l!}$ we may switch the following sum and limit and obtain 
\begin{equation}\label{limit1}
\lim_{n\to \infty}\sum_{l=N+1}^\infty a_l(x_n)(1-r_n)^{l-N}=0.
\end{equation}
As $X$ is compact, we may w.l.o.g. assume that $x_n\to x\in X$ converges. In particular, (\ref{limit2}) and (\ref{limit1}) imply that $$0=\lim_{n\to \infty}(1-r_n)^N\sum_{l=0}^N a_l(x_n)(1-r_n)^{l-N}=\lim_{n\to \infty}a_0(x_n).$$
Therefore, one obtains that
$$0=\lim_{n\to \infty}(1-r_n)\sum_{l=1}^Na_l(x_n)(1-r_n)^{l-N}=\lim_{n\to \infty}\sum_{l=0}^{N-1}a_l(x_n)(1-r_n)^{l-(N-1)}.
$$
Repeating this argument implies that $a_l(x_n)\to 0$, for $0\leq l\leq N$. Since the functions $a_n$ are continuous, it follows that $a_l(x)=0$, for all $0\leq l\leq N$. In particular, there exists $N^\prime>N $ such that $x\in A_{N^\prime}$. This contradicts the maximality of $N$.
\end{proof}
\begin{thm}\label{thmasymp}
For $\mu\in \a^\ast$, there exists $N\geq 0$ such that
\begin{align*}
\limsup_{t\to 1}\left(\left(1-t\right)^N\sup_{(h,k)\in \del\Omega_\p\times K}|\alpha(\exp(-ith)k)^{\mu}|\right)<\infty.
\end{align*}
\end{thm}
\begin{proof}
Let $\{\omega_1,\dots,\omega_r\}\subseteq \a^\ast$ be the basis of fundamental spherical weights. Then there exist $c_i\in \R$ such that $\mu=c_1\omega_1+\dots c_n\omega_n$ and $|\alpha(\exp(-ith)k)^{\omega_j}|$ is bounded from above (cf. Proposition \ref{formulaalpha}), the factors corresponding to $c_j>0$ are negligible. Hence it suffices to prove the assertion for $\mu=-\omega$, where $\omega$ is a fundamental weight. Let $(\pi,V)$ be the corresponding irreducible finite-dimensional spherical representation. Since $\pi$ is irreducible, it is a highest weight representation and thus the results from this subsection are applicable. We may assume that the supremum is taken over $\del\Omega_\a\times K$ since $\Ad(K)\del\Omega_\a=\del\Omega_\p$ and $\alpha(gk_2)=\alpha(k_1^{-1}gk_1k_2)$, for $g\in G$ and $k_1,k_2\in K$. Consider
$$X:= \del\Omega_\a \times K\qand f_{x,k}(t):= |\alpha\exp(-itx)k)^{2\omega}|^2=\sum_{n=0}^\infty a_n(x,k)(1-t)^n\neq 0,$$
for $t\in (0,1)$. Lemma \ref{lemN0} implies that there is an $N\in \N_0$ maximal such that  
$$A_N=\{(x,k)\in \del\Omega_\p\times K\;| a_n(x,k)=0\qforall 0\leq n<N,\quad a_{N}(x,k)\neq 0\}
$$
is non-empty and Lemma \ref{coeff} implies $\|a_n\|_\infty\leq \frac{(2C_{\rm max})^n}{n!}$. Therefore, the assertion follows from the previous proposition.
\end{proof}
\begin{thm}\label{ThmMainResult}
Let $G$ be a connected semisimple Lie group. Then there exists $N>0$ such that the functions
\begin{align*}
{\rm(i)}\;&\Omega_\p\ni x\mapsto \sup_{k\in K} s_{\rm max}^{G_\C}(\kappa(\exp(-ix)k)),\\
{\rm(ii)}\;&\Omega_\p\ni x\mapsto \sup_{k\in K}s_{\rm max}^{G_\C}(\alpha(\exp(-ix)k)),\\
{\rm(iii)}\;&\Omega_\p\ni x\mapsto \sup_{k\in K}s_{\rm max}^{G_\C}(\eta(\exp(-ix)k)),
\end{align*}
are bounded by a constant multiple of $\left(\pitwo-r_{\rm spec}(\ad(x))\right)^{-N}$.
\end{thm}
\begin{proof}
Since $\exp(-ix)k\in U$, for $x\in \Omega_\p$ and $k\in K$, one has $s_{\rm max}^{G_\C}(\exp(-ix)k)=1$. In particular, Proposition \ref{propScales} implies that there exist $C^\prime,N^\prime>0$ such that 
$$
s_{\rm max}^{G_\C}(\kappa(\exp(-ix)k)),\quad s_{\rm max}^{G_\C}(\eta(\exp(-ix)k))\leq C^\prime s_{\rm max}^{G_\C}(\alpha(\exp(-ix)k))^{N^\prime},$$
for all $x\in \Omega_\p$ and $k\in K$. Therefore, it suffices to prove (ii). Recall that
$$s_{\rm max}^{G_\C}(H_1+iH_2)=e^{r_{\rm spec}(\ad(H))},\qfor H_1,H_2\in\a.$$
For $c_i>0$ sufficiently large, one has $r_{\rm spec}(H)\leq \sum_{i=1}^n c_i|\omega_i(H)|$, for all $H\in \a$,
and thus the assertion follows from Theorem \ref{thmasymp}.
\end{proof}
\section{Applications of the growth estimates}\label{secAppl}
Let $G$ be a connected semisimple Lie group with not necessarily finite center and let $(\pi_\sigma,\cH_\sigma)$ be the principal series representation (cf. (\ref{defPrinc1})) associated to a finite-dimensional representation of $P_{\rm min}$, which is unitary when restricted to $M$. The main results of this section are the following:
\begin{itemize}
\item[(i)] For $v\in \cH_\sigma^{[K]}$, there exist $C,N>0$ such that
\begin{equation}\label{EqPrincGrowth}
\|e^{i\del\pi_\sigma(x)}v\|\leq C\left(\pitwo-r_{\rm spec}(\ad(x))\right)^{-N},
\end{equation}
for all $x\in \Omega_\p$, which proves (ii) in Theorem \ref{thm1}. 
\item[(ii)] For $\lambda\in \cH_\sigma^{-\infty}$ and $v\in \cH_\sigma^{[K]}$, there exist $C,N>0$ such that
\begin{equation}\label{EqDistrVgrowth}
\lambda(e^{i\del\pi_\sigma(x)}v)\leq C\left(\pitwo-r_{\rm spec}(\ad(x))\right)^{-N},
\end{equation}
for all $x\in \Omega_\p$. If $G$ has finite center, then this implies Theorem \ref{thm3} (cf. Remark \ref{remModGrowth}, Corollary \ref{corModGrowth}).
\item[(iii)] If $G$ has finite center, (i) and (ii) hold for arbitrary Hilbert globalizations of Harish--Chandra modules proving the remaining part of Theorem \ref{thm1}.
\end{itemize}
\subsection{Growth of Poisson transforms of distributional sections}\label{subsPoiss}
We slightly improve upon Theorem \ref{ThmMainResult} by showing that the asymptotic of the $K$-derivatives of the Iwasawa component maps in finite-dimensional representations of $G$ is also polynomial. We then apply this to Poisson transforms of distributional sections.
\begin{defin}
\rm{Let $V$ be a finite-dimensional  complex vector space and let $\{X_1,\dots,X_n\}$ be a basis of $\k$. For a multiindex $\ell\in \N_0^n$, we denote $|\ell|:=\ell_1+\dots +\ell_n$ as well as
$$(L_X^\ell f)(k):=\left.\frac{\del^{|\ell|}}{\del t_1^{\ell_1}\dots \del t_n^{\ell_n}}\right\vert_{t=0} f(\exp(-(t_1X_1+\dots+t_nX_n))k),\qfor f\in C^\infty(K,V).$$
}
\end{defin}
\begin{lem}
Let $M$ be a manifold. Consider smooth functions $f_j:M\to \C$ and an open submanifold $M_0\subseteq M$ such that $f_j(M_0)$ is contained in a simply connected subset $U\subseteq\C^\times$, which guarantees that
$$F:M_0\to \C,\quad m\mapsto \prod_{j=1}^n f_j(m)^{\mu_j},\qfor \mu_j\in \C,
$$
is a well-defined smooth function. Let $D$ be a differential operator on $M$. Then, the function $DF:M_0\to \C$ is contained in the $C^\infty(M)$-module generated by products of
$$f_j^{\mu_j-k_j}:M_0\to \C,\quad m\mapsto f_j(m)^{\mu_j-k_j},\qfor 1\leq j\leq n,\quad k_j\in \N_0.
$$
\end{lem}
\begin{proof}
This is an immediate consequence of the Leibniz rule and the chain rule since the derivatives of the functions $f_j$ are contained in $C^\infty(M)$.
\end{proof}
\begin{prop}\label{propAsympDiffA}
For $\mu \in \a^\ast$ and $\ell\in \N_0^n$, there exist $C_{\mu,\ell},N_{\mu,\ell}>0$ such that
$$\sup_{k\in K} |L_X^\ell\alpha(\exp(-ix)k)^\mu|\leq C_{\mu,\ell} \left(\pitwo -r_{\rm spec}(\ad(x))\right)^{-N_{\mu,\ell}},\qforall x\in \Omega_\p.$$
\end{prop}
\begin{proof}
We may w.l.o.g. assume that $x\in \Omega_\a$ since $\Omega_\p=\Ad(K)\Omega_\a$ and $\alpha$ is left $K$-invariant. For $\mu_j\in \R$ such that $\mu=\sum_{j=1}^r \mu_j\omega_j$, one has
\begin{equation}\label{alpha}
\alpha(\exp(-ix)k)^\mu = \prod_{j=1}^r \left(\alpha(\exp(-ix)k)^{2\omega_j}\right)^{\mu_j/2}.
\end{equation}
We apply the previous lemma to $M=\a_\C\times K$ and $M_0=(\a+i\Omega_\a)\times K$ as well as the analytic functions
$$f_j:\a_\C\times K\to \C,\quad (x,k)\mapsto \alpha(\exp(-ix)k)^{2\omega_j}=\sum_{\mu\in \cP_j}e^{-2i\mu(x)}\|v_\mu(k)\|^2
$$
(cf. Proposition \ref{formulaalpha}). In particular, the function 
$$M_0\to \C,\quad (x,k)\mapsto L_X^\ell \alpha(\exp(-ix)k)^\mu$$
is contained in the $C^\infty(M)$-module generated by products of $f_j^{\mu_j-k_j}:M_0\to \C$. Since $\del\Omega_\p \times K/Z(G)$ is compact, the coefficient functions in $C^\infty(M)$ do not contribute to the asymptotic and thus the assertion follows from Theorem \ref{ThmMainResult}.
\end{proof}
\begin{cor}\label{corAsympDiffN}
Let $(\sigma,W_\sigma)$ be a nilpotent finite-dimensional representation of $N$. For $\ell\in \N_0^n$, there exist $C_{\sigma,\ell},N_{\sigma,\ell}>0$ such that
$$\sup_{k\in K} \|L_X^\ell\sigma(\eta(\exp(-ix)k))\|\leq C_{\sigma,\ell}\left(\pitwo -r_{\rm spec}(\ad(x))\right)^{-N_{\sigma,\ell}},\qforall x\in \Omega_\p.
$$
\end{cor}
\begin{proof}
Corollary \ref{corUnip} implies that $\sigma^{v,w}\circ \eta\in \C[\KANC]$, for all $v,w\in W_\sigma$, and one has
\begin{align*}
\C[\KANC]=\left\{\frac{p}{\Delta^k}\;|\;p\in \C[G_\C],\quad k\in \N_0 \right\}
\end{align*}
as well as $\Delta(g)=\alpha(g)^\mu$, for $g\in \KANC$ and $\mu=\sum_{i=1}^r \omega_i\in \a^\ast$ (cf. (\ref{Delta})). By the same argument as before using the compactness of $\del\Omega_\p\times K/Z(G)$, the polynomials in the numerator do not contribute to the asymptotic on the boundary and thus the assertion follows immediately from the previous lemma.
\end{proof}
\begin{cor}\label{corAsympDiffK}
Let $(\tau,V_\tau)$ be a finite-dimensional unitary representation of $K$. For every $\ell\in \N_0^n$, there exist $C_{\tau,\ell},N_{\tau,\ell}>0$ such that
$$\sup_{k\in K} \|L_X^\ell\tau(\kappa(\exp(-ix)k))\|\leq C_{\tau,\ell}\left(\pitwo -r_{\rm spec}(\ad(x))\right)^{-N_{\tau,\ell}},\qforall x\in \Omega_\p.
$$
\end{cor}
\begin{proof}
We first prove the assertion for its linearization $G_\R\subseteq G_\C$ and then argue that the general case can be reduced to the linear case. If $G=G_\R$ is linear, then $$K_\R \subseteq U=\langle \exp_{G_\C}(\k\oplus i\p)\rangle$$
is a compact subgroup. Applying Frobenius reciprocity to $\Ind_{K_\R}^{U}(\tau)$, for $\tau\in \widehat{K}_\R$, shows that every irreducible unitary representation of $K_\R$ occurs as the restriction of a subrepresentation of an irreducible finite-dimensional unitary representation of $U$. Furthermore, by Weyl's Unitarian Trick, every finite-dimensional representation of $U$ extends to $G_\C$ and thus there exists a finite-dimensional representation $(\pi,V)$ of $G_\C$ which contains a $K_\R$-invariant subspace which is equivalent to the representation $(\tau,V_\tau)$ of $K_\R$. Observe that
$$\pi(\kappa(x^{-1}k))=\pi(x^{-1}k)(\pi(\alpha(x^{-1}k)\pi(\eta(x^{-1}k)))^{-1},\qfor x\in \xig\qand k\in K_\R.
$$
Therefore, Proposition \ref{propAsympDiffA} and Corollary \ref{corAsympDiffN} together with
$$\|\tau(\kappa(x^{-1}k))\|\leq \|\pi(\kappa(x^{-1}k))\|\leq \|\pi(x^{-1}k)\|\|\pi(\alpha(x^{-1}k)))^{-1}\|\|\pi(\eta(x^{-1}k))^{-1}\|
$$
imply the assertion in the linear case. We assume w.l.o.g. that $G$ is a simply connected simple Lie group and distinguish two cases. If $G$ is not hermitian, then $K$ is a semisimple compact Lie group. If $G$ is hermitian, then \cite[App. 11.A.2 p. 115]{Wa92} implies that
$$Z(K)_e\cong \R\qand Z:=\ker(\eta_G)\cong \Z.$$
In the first case, $G$ is a finite cover of a linear Lie group and thus there exists $m\in \N$ such that the symmetrized tensor product $\tau^{\vee m}$ factors through a unitary representation of $K_\R$. We now argue inductively over the degree $|\ell|=\ell_1+\dots +\ell_n$ of $\ell\in \N_0^n$. For $|\ell|=0$, the assertion follows from Theorem \ref{ThmMainResult}. The induction hypothesis allows us to ignore lower-order terms and thus we consider
$$L_X^\ell \tau^{\vee m}(\kappa(x^{-1} k))=m \left(L_X^\ell\tau(\kappa(x^{-1}k))\right)\vee \tau^{\vee (m-1)}(\kappa(x^{-1}k))+\text{lower order terms}.
$$
In particular, one obtains that
$$\|L_X^\ell\tau(\kappa(x^{-1}k))\|\lesssim \frac{1}{\|\tau^{\vee (m-1)}(\kappa(x^{-1}k))\|}\left(\|L_X^\ell \tau^{\vee m}(\kappa(x^{-1} m))\|+\text{lower order terms} \right).
$$
Since $\tau^{\vee m}$ factors through a representation of $K_\R$, the asymptotic of its derivatives reduces to the linear case. Furthermore, one has $\frac{1}{\|S\|}\leq \|S^{-1}\|$, for invertible endomorphisms and thus, the assertion follows inductively from
$$\|\tau^{\vee (m-1)}(\kappa(x^{-1}k))^{-1}\|\leq s_{\rm max}^{K_\C}(\kappa(x^{-1}k))^{m-1}.$$
For hermitian groups, $K\cong \R\times [K,K]$ and the universal complexification $\eta_G:G\to G_\C$ is injective on $[K,K]$ (cf. (1) in the proof of \cite[Lemma 11.A.2.3]{Wa92}). Any irreducible unitary representation of $\R\times [K,K]$ is a tensor product of an irreducible unitary representation of $\R$, i.e. a character, and an irreducible unitary representation of $[K,K]$. Since $\eta_G$ is injective on $[K,K]$ the asymptotic of the derivatives of irreducible unitary representations of $[K,K]$ reduces to the linear case. Therefore, it suffices to show the assertion for characters. Since one can scale characters such that they factor through characters of $K_\R$, the assertion follows by the same inductive argument as above. 
\end{proof}
\begin{rem}\label{remModGrowth}
\rm{Let $G$ be a connected semisimple Lie group and $f:G\to \C$ be a $(Z(\g),K)$-finite smooth function. Theorem \ref{ThmHolExt} implies that $f$ extends to a holomorphic function $f:\wxi\to \C$. The function $f$ is called of \textit{moderate growth} if there exists a constant $d>0$ such that
\begin{equation}\label{defModGrowth}
\sup_{g\in G} \frac{|L_{D}(f)(g)|}{\|g\|^d}<\infty,\qforall D\in \cU(\g),
\end{equation}
where $\|k\exp(X)\|:=e^{\|X\|}$, for $k\in K$, $X\in \p$ and a $K$-invariant norm on $\p$. If $G$ has finite center, then \cite[Thm. 11.9.2]{Wa92} implies that any $(Z(\g),K)$-finite smooth function $f:G\to \C$ of moderate growth has a ``Poisson integral representation'' in the sense that there exists a finite-dimensional $P_{\rm min}$-representation $(\sigma,W_\sigma)$, which is unitarizable when restricted to $M$, a distribution vector $\lambda\in \cH_\sigma^{-\infty}$ of the corresponding principal series representation $(\pi_\sigma,\cH_\sigma)$ (cf. Subsection \ref{subsHC}) and a $K$-finite vector $v\in \cH_\sigma^{[K]}$ such that
\begin{equation}\label{EqDistrSec}
f(g)=\pi^{\lambda,v}(g):=\lambda(\pi_\sigma(g)v),\qforall g\in G.
\end{equation}
The Kr\"otz--Stanton Extension Theorem for Fr\'{e}chet globalizations of moderate growth (cf. Corollary \ref{corKrSt}) implies that $$\pi_\sigma^v(\wxi)\subseteq \cH^\infty\qand \pi_\sigma^v:\wxi\to \cH^\infty$$
is a holomorphic Fr\'{e}chet space valued function. In particular, since $\lambda$ is antilinear, the antiholomorphic extension of $f=\pi^{\lambda,v}$ is given by
$$f(x)=\lambda(\pi_\sigma^v(x)),\qfor x\in \wxi.$$
We now apply the results of this section to obtain polynomial growth rates for such functions as $x$ approaches the boundary of the crown.
}
\end{rem}
\begin{prop}\label{propGrowth}
Let $(\sigma,W_\sigma)$ be a finite-dimensional representation of $P_{\rm min}$, which is unitarizable when restricted to $M$. For a $K$-finite vector $v\in \cH_\sigma^{[K]}$ and $\ell\in \N_0^n$, there exist constants $C_{\sigma,\ell},N_{\sigma,\ell}>0$ such that
\begin{equation}\label{EqGrowthPiSigma}
\sup_{k\in K} \|(L_X^\ell\pi_\sigma^v(\exp(ix)))(k)\|\leq C_{\sigma,\ell} \cdot \left(\pitwo- r_{\rm spec}(\ad(x))\right)^{-N_{\sigma,\ell}},\qforall x\in \Omega_\p.
\end{equation}
\end{prop}
\begin{proof}
The $K$-finite $M$-right equivariant function $v:K\to W_\sigma$ extends to a holomorphic function $v:\wt{K}_\C\to W_\sigma$ and
$$\pi_\sigma^{v}(\exp(ix))(k)=\sigma_\rho(\beta(\exp(-ix)k))^{-1}v(\kappa(\exp(-ix)k)),
$$
for $x\in \Omega_\p$. In view of the Leibniz rule, it suffices to prove that
$$\sup_{k\in K}\|L_X^\ell\sigma_\rho(\beta(\exp(-ix)k))^{-1}\|\qand \sup_{k\in K}\|L_X^\ell v(\kappa(\exp(-ix)k))\|
$$
satisfy estimates of the form (\ref{EqGrowthPiSigma}). Since $v$ is a $K$-finite vector, there exists a not necessarily irreducible unitary finite-dimensional representation $\tau$ of $K$ such that 
$$\|L_X^\ell v(\kappa(\exp(-ix)k))\|\leq \|L_X^\ell \tau(\kappa(\exp(-ix)k))\|,
$$ 
which can be estimated by Corollary \ref{corAsympDiffK}. The restriction of the representation $d\sigma_\rho$ to $\a_\C\oplus \n_\C$ decomposes into generalized $\a_\C$-eigenspaces, which are nilpotent $\n_\C$-modules. As an $\a_\C$-module, the generalized $\a_\C$-eigenspaces are a tensor product of a one-dimensional representation and a nilpotent representation. The contribution of the $\a_\C$-representation can be estimated by Proposition \ref{propAsympDiffA} and the nilpotent $N_\C$-representation can be estimated by Corollary \ref{corAsympDiffN}.
\end{proof}
Before we present our main result, we recall the some facts about the dual space of the smooth vectors $\cH_\sigma^\infty$ of a principal series representation.
\begin{rem}
\rm{Let $(\sigma,W_\sigma)$ be a finite-dimensional $P_{\rm min}$-representation, which is unitarizable when restricted to $M$. The dual space 
$$\left(\cH_\sigma^{\infty}\right)^\prime=\left\{\overline{\lambda}:\cH_\sigma^\infty\to \C\;|\;\lambda\in \cH_\sigma^{-\infty}\right\}$$ 
can be identified with the space of distributional sections $\Gamma^{-\infty}(\W_\sigma)$ of the associated vector bundle $\W_\sigma:=K\times_M W_\sigma$ and the space of smooth sections $\Gamma^\infty(\W_\sigma)$ of $\W_\sigma\epi K/M$ can be identified with the smooth $M$-equivariant maps $C^\infty(K,W_\sigma)^M$. The map
$$\Phi:C_c^\infty(K,W_\sigma)\to C^\infty(K,W_\sigma)^M,\quad f\mapsto f^M,\qfor f^M(k):=\int_M \sigma(m)f(km)dm
$$
is surjective and one equips $C^\infty(K,W_\sigma)^M$ with the corresponding quotient topology induced from the canonical LF-topology on $C_c^\infty(K,W_\sigma)$. Since $K/M$ is compact, the quotient topology on $C^\infty(K,W_\sigma)^M$ induced by $\Phi$ is a Fr\'{e}chet topology and is induced by the seminorms
$$p_\ell(f):=\sup_{k\in K} \|L_X^\ell f(k)\|,\qfor f\in C^\infty(K,W_\sigma)^M\qand \ell\in \N_0^n.
$$
For $m\in M$ and $k\in K$, one has that
$$\|L_X^\ell f(km)\|=\|\sigma(m)^{-1}(L_X^\ell f(k))\|=\|L_X^\ell f(k)\|$$
and since $K/M$ is compact, the seminorms $p_\ell $ are well-defined. Consider the dual space 
$$\Gamma^{-\infty}(\W_\sigma):=(\Gamma^{\infty}(\W_\sigma))^\prime$$ 
with respect to this Fr\'{e}chet topology. Its elements are called \textit{distributional sections}. The order of a distributional section $T\in \Gamma^{-\infty}(\W_\sigma)$ is the minimal $N\in \N_0$ such that
$$|T(f)|\lesssim \sum_{|\ell|\leq N} p_\ell(f),\qforall f\in \Gamma^{\infty}(\W_\sigma).
$$
If $T$ is a distributional section of order $0$, then the density of $C^\infty(K/M)\subseteq C(K/M)$ and the estimate $|T(f)|\lesssim \|f\|_\infty$ implies that $T\circ \Phi$ extends to a continuous linear functional on $C_c(K,W_\sigma)$ and thus the Riesz Representation Theorem implies that there exists a $W_\sigma^\ast$-valued  measure $\mu_T$ such that $(R_m)^\ast\mu_T=\sigma^\ast(m)^{-1}\mu_T$ and
$$T(f)=\int_{K/M}d\mu_T(k)(f(k)),\qforall f\in \Gamma(\W_\sigma).
$$
We proceed as in the proof of \cite[Thm. 24.4]{Tre67} to show that if $T$ is a distributional section of finite order that $T$ is a sum of derivatives of distributional sections of order zero. Let $T$ be a distributional section of order $N$. Let $\Gamma^N(\W_\sigma)$ be the space of $N$-times continuously differentiable sections and $\Gamma(\W_\sigma)$ be the space of continuous sections equipped with their natural Fr\'{e}chet topologies. Consider the continuous linear embedding with closed image
$$\Gamma^N(\W_\sigma)\to \prod_{|\ell|\leq N}\Gamma(\W_\sigma),\quad f\mapsto (L_X^\ell f)_{|\ell|\leq N}.
$$
Since $T$ is a distributional section of order $N$, it extends to a continuous linear functional on $\Gamma^N(\W_\sigma)$ and thus on the image of the embedding. The Hahn--Banach Extension Theorem implies that $T$ extends to a continuous linear functional on $\prod_{|\ell|\leq N}\Gamma(\W_\sigma)$ and thus there exist $W_\sigma^\ast$-valued measures $\mu_{\ell,T}$, for $|\ell|\leq N$ such that $(R_m)^\ast\mu_{\ell,T}=\sigma^\ast(m)^{-1}\mu_{\ell,T}$ and 
\begin{equation}\label{Radon}
T(f)=\sum_{|\ell|\leq N} \int_{K/M}d\mu_\ell(k)(L_X^\ell f)(k),\qforall f\in \Gamma^N(\W_\sigma).
\end{equation}
}
\end{rem}
\begin{thm}\label{thmModGrowth}
For $\lambda\in \cH_\sigma^{-\infty}$ and $v\in \cH_\sigma^{[K]}$, there exist $C,N>0$ such that
$$|\lambda(e^{i\del\pi_\sigma(x)}v)|\leq C\cdot \left(\pitwo-r_{\rm spec}(\ad(x))\right)^{-N},\qforall x\in \Omega_\p.
$$
\end{thm}
\begin{proof}
Let $N$ be the order of $\lambda\in \cH_\sigma^{-\infty}$ and let $\mu_\ell $ be $W_\sigma^\ast$-valued Radon measures such that
$$\overline{\lambda(w)}=\sum_{|\ell |\leq N} \int_{K/M}d\mu_\ell (k)(L_X^\ell w)(k),\qforall w\in \Gamma^N(\W_\sigma).
$$
In particular, for the finite measures $\|\mu_\ell \|$ on $K/M$, one has
\begin{align*}
|\lambda(e^{i\del\pi_\sigma(x)}v)|\leq \sum_{|l|\leq N} \|\mu_\ell \|(K/M) \sup_{k\in K}\|(L_X^\ell\pi_\sigma^v(x))(k)\|.
\end{align*}
Therefore, the assertion follows from Proposition \ref{propGrowth}.
\end{proof}
As a consequence of Remark \ref{remModGrowth} and the previous theorem, we obtain the following:
\begin{cor}\label{corModGrowth}
Let $G$ have finite center and let $f:G\to \C$ be a $(Z(\g),K)$-finite smooth function of moderate growth. Then, there exist $C,N>0$ such that
$$|f(\exp(ix))|\leq C\cdot \left(\pitwo-r_{\rm spec}(\ad(x))\right)^{-N},\qforall x\in \Omega_\p.
$$
\end{cor}
\begin{rem}
\rm{Let $(\pi,\cH)$ be a Hilbert globalization of a Harish--Chandra module. For $v\in \cH$, consider the kernel
$$K^{v,v}:G\times G\to \C,\quad (g_1,g_2)\mapsto \langle \pi(g_1)v,\pi(g_2)v\rangle.
$$
Then, $K^{v,v}$ is $(Z(\g\oplus \g),K\times K)$-finite. Consider the unique antihomomorphism $$\cU(\g)\to \cU(\g),\quad D\mapsto D^t,\qsuchthat x^t=-x,\qfor x\in \g.$$
For $D_1,D_2\in \cU^{\leq k}(\g)$ and $g_1,g_2\in G$, one has
$$L_{D_1,D_2}K^{v,v}(g_1,g_2)=\langle d\pi(D_1^t)\pi(g_1)v,d\pi(D_1^t)\pi(g_1)v\rangle.
$$
Since $\pi$ is a Hilbert representation, it follows from \cite[Lemma 2.A.2.2]{Wa88} that there exist $C>0$ and $d>0$ such that
$$|L_{D_1,D_2}K^{v,v}(g_1,g_2)|\leq C\|g_1\|^d\|g_2\|^d \|d\pi(\Ad(g_1)^{-1}D_1^t)v\|\|d\pi(\Ad(g_2)^{-1}D_2^t)v\|.
$$
Since $\cU^{\leq k}(\g)$ is a finite-dimensional $G$-module, it follows from the same argument as in \cite[Lemma 11.5.1]{Wa92} that $\|d\pi(\Ad(g_i)^{-1}D_i^t)v\|$ is bounded by a constant multiple of $\|g_i\|^{d_i}$ and thus $K^{v,v}$ has moderate growth for the connected semisimple Lie group $G\times G$. 
}
\end{rem}
\begin{thm}\label{thmFinCenter}
Let $G$ have finite center and let $(\pi,\cH)$ be a Hilbert globalization of a Harish--Chandra module. For $v\in \cH^{[K]}$, there exist $C,N>0$ such that
$$\|e^{i\del\pi(x)}v\|\leq C\left(\pitwo-r_{\rm spec}(\ad(x))\right)^{-N},\qforall x\in \Omega_\p.
$$
\end{thm}
\begin{proof}
The kernel $K^{v,v}$ is $(Z(\g\oplus \g),K\times K)$-finite and extends to a holomorphic function $K^{v,v}:\wxi\times \wxi\to \C$ by Proposition \ref{PropMatrixCoeff}. Furthermore, $K^{v,v}$ is of moderate growth for the semisimple Lie group $G\times G$ by the previous remark. For $x_0\in \del\Omega_\p$, the functions
$\cS_{\pm 1}\to \C$ defined as
$$z\mapsto \langle e^{-\overline{z}\del\pi(x_0)}v,e^{z\del\pi(x_0)}v\rangle\qand z\mapsto K^{v,v}(\exp(-zx),\exp(zx))
$$
are holomorphic and coincide on $\R$. Therefore, they coincide on $\cS_{\pm 1}$, which implies
$$\|e^{i\del\pi(x)}v\|^2=K^{v,v}(\exp(-ix),\exp(ix))
,\qforall x\in \Omega_\p.$$
Remark \ref{remModGrowth} implies that $K^{v,v}$ is a Poisson transform of a distributional section in the sense of (\ref{EqDistrSec}) and thus the assertion follows from Theorem \ref{thmModGrowth}.
\end{proof}
We now discuss how to transfer these growth estimates to growth estimates of derivatives.
\begin{lem}\label{lemKNstatement}
Let $\cH$ be a Hilbert space, $\D:=\{z\in \C\;|\;|z|<1\}$ and $N>0$. Suppose $f:\D\to \cH$ is holomorphic and $$\|f(z)\|\leq (1-|z|)^{-N},\qforall z\in \D.$$
Then, there exists $C>0$ independent of $f$ such that
$$\frac{1}{k!}\left\|f^{(k)}(0)\right\|\leq C(1+k)^N\quad\text{as}\quad k\to \infty.
$$
\end{lem}
\begin{proof}
For $r\in (0,1)$, Cauchy's Integral Formula implies
$$\frac{1}{k!}\left\|f^{(k)}(0)\right\|\leq  \int_0^{2\pi} \frac{\|f(re^{i\theta})\|}{r^{k}}\frac{d\theta}{2\pi} \leq F_k(r):=(1-r)^{-N} r^{-k}.
$$
Minimizing the right hand side leads to the choice $r_k:=\frac{k}{N+k}\in (0,1)$ and one has
\begin{align*}
F_k(r_k)=&\left(1-\tfrac{k}{N+k}\right)^{-N}\left(\tfrac{k}{N+k}\right)^{-k}=\left(\tfrac{N}{N+k}\right)^{-N}\left(\tfrac{k}{N+k}\right)^{-k}=\left(1+\tfrac{k}{N}\right)^N\left(1+\tfrac{N}{k}\right)^k.
\end{align*}
Therefore, $\left(1+\frac{k}{N}\right)^N\leq (1+k)^N$ and $\left(1+\frac{N}{k}\right)^k\to e^N$ imply the assertion.
\end{proof}
As an immediate consequence of the previous lemma, Theorem \ref{thmModGrowth} and Theorem \ref{thmFinCenter}, we obtain the following theorem:
\begin{thm}\label{thmNormGrowth}
Let $G$ have finite center and let $(\pi,\cH)$ be a Hilbert globalization of a Harish--Chandra module. 
\begin{itemize}
\item[{\rm (i)}] Let $v\in \cH^{[K]}$. Then, there exist $C,N>0$ such that
$$\frac{1}{k!}\|d\pi(x_0)^kv\|\leq C(1+k)^N,\qforall x_0\in \del\Omega_\p.
$$
\item[{\rm (ii)}] Let $\lambda\in \cH^{-\infty}$ and $v\in \cH^{[K]}$. Then, there exist $C,N>0$ such that
$$\frac{1}{k!}\left|\lambda(d\pi(x_0)^kv)\right|\leq C(1+k)^N,\qforall x_0\in \del\Omega_\p.
$$
\end{itemize}
\end{thm}
\subsection{Dependence of the constants on the representation}
We now prove estimates of the form
$$\|e^{i\del\pi_{\sigma,\mu}(x)}v\|\leq C \left(\pitwo-r_{\rm spec}(\ad(x))\right)^{-N},\qforall x\in \Omega_\p,
$$
for principal series representations and explicitly determine how the constant $C>0$ and the exponent $N>0$ depends on the representation $[\sigma]\in \widehat{M}$, the functional $\mu\in \a_\C^\ast$ and the $K$-type $[\tau]\in \widehat{K}$ corresponding to the $K$-finite vector $v\in \cH_{\sigma,\mu}^{[K]}$. We apply this in Subsection \ref{subsTemp} to direct integrals of irreducible unitary representations, for which the dependency of these estimates on the exponent is relevant.
\begin{prop}
Let $(\tau,V_\tau)$ be an irreducible unitary representation of $K$, let $\t\subseteq \k$ be a maximal torus and choose a $\cW_\k$-invariant norm $|\cdot |$ on $i\t^\ast$. Then there exists $c>0$ independent of $\tau$ such that
$$\|\tau(k)\|\leq  s_{\rm max}^{G_\C}(k)^{c|\tau|},\qforall k\in K_\C.
$$
\end{prop}
\begin{proof}
Let $\t\subseteq \k$ be a maximal torus. For $X\in \k$, choose $k_0\in K$ for which $X_0:=\Ad(k_0)X$ is contained in $\t$. In particular,
$$\|\tau(k)e^{d\tau(iX)}\|=\|e^{id\tau(X_0)}\|=\max_{\lambda\in \cP_\tau} e^{i\lambda(X_0)},\qfor k\in K,
$$
where $\cP_\tau\subseteq i\t^\ast$ is the set of $\t$-weights of $\tau$. Recall that $\cP_\tau\subseteq \overline{\text{conv}(\cW_\k.\lambda_\tau)}$, for the highest weight $\lambda_\tau\in \cP_\tau$ w.r.t. some positive system $\Delta^+\subseteq \Delta(\k_\C,\t_\C)$. In particular, for any choice of norm $\|\cdot\|$ on $i\t^\ast$ and $|\tau|:=\|\lambda_\tau\|$, it follows that
$$\max_{\lambda\in \cP_\tau}|i\lambda(X_0)|\leq c|\tau|\rho(iX_0),\qfor \rho(iX_0):=r_{\rm spec}(\ad(iX_0)).
$$
The constant $c>0$ is independent of $\tau$ and only depends on the choice of the norm $|\cdot |$ on $i\t^\ast$. Therefore the assertion follows from
$s_{\rm max}^{G_\C}(k\exp(iX))=e^{\rho(iX)}$, for $k\in K$ and $X\in \k$.
\end{proof}
\begin{cor}\label{corKgrowth}
There exists $C_0>0$ such that for all $$[\sigma]\in \widehat{M},\quad [\tau]\in\widehat{K}\qand f\in \cH_{\sigma,[\tau]}^{[K]},$$ 
where $\Ind_M^K(\sigma)=(\pi_\sigma,\cH_\sigma)$ is the corresponding induced representation, one has
$$\|f(k)\|\leq C_0\cdot d_\tau^3\cdot s_{\rm max}^{K_\C}(k)^{|\tau|}\cdot \|f\|,\qforall k\in \wt{K}_\C.
$$
\end{cor}
\begin{proof}
By Frobenius reciprocity, all irreducible unitary representations of $K$ have finite multiplicity in $\cH_\sigma= L^2(K\times_\sigma W)$. We choose a Haar measure on $K/Z(G)$ such that
$$\int_{K/Z(G)}|\chi_\tau(k)|^2dk=1,\qforall [\tau]\in \widehat{K}\qand \chi_\tau(k):=\tr(\tau(k)).
$$
For $[(\tau,V_\tau)]\in \widehat{K}$, consider the map $\Phi_{[\tau]}:V_\tau\otimes \Hom_M(V_\tau,W)\to \cH_{\sigma,[\tau]},\;v\otimes T\mapsto f_{T,v}$, where
\begin{equation}\label{Frob}
f_{T,v}:K\to W,\quad k\mapsto T(\tau(k)^{-1}v),\qfor T\in \Hom_M(V_\tau,W)\qand v\in V_\tau.
\end{equation}
Frobenius reciprocity also implies that $\Phi_{[\tau]}$ is a linear isomorphism. In particular, for $d_\tau:=\dim_\C V_\tau$, the intertwining operator $A_\tau\in \End_K(V_\tau)=\C\mathds{1}$ defined as 
$$A_\tau:=\int_{K/Z(G)}\chi_\tau(k)\tau(k)^{-1}dk\quad\text{satisfies}\quad \tr(A_\tau)=\int_{K/Z(G)}|\chi_\tau(k)|^2dk=1.
$$
Thus, $d_\tau A_\tau=\mathds{1}$ combined with (\ref{Frob}) implies that
$$f(e)=d_\tau\int_{K/Z(G)} \chi_\tau(k)f(k)dk,\qforall f\in \cH_{\sigma,[\tau]}.
$$
Let $\{w_1,\dots,w_n\}\subseteq W$ be an orthonormal basis and define
$$f_i:K\to W,\quad f_i(k):=d_\tau\int_{M/Z(G)}\overline{\chi_\tau(km)} \sigma(m)w_idm.
$$
Then, $f_i(km)=\sigma(m)^{-1}f_i(k)$ implies $f_i\in \cH_\sigma$ and $\|\chi_\tau\|_\infty\leq d_\tau$ implies $\|f_i\|_\infty\leq d_\tau^2$. In particular,
\begin{align*}
&\sum_{i=1}^n \int_{K/M} \langle f_i(k),f(k)\rangle dkM\cdot w_i \\
=&\sum_{i=1}^n \int_{K/M}\int_{M/Z(G)} \left\langle \overline{\chi_\tau(km)} \sigma(m)w_i,f(k) \right\rangle dmdkM\cdot w_i\\
=&\sum_{i=1}^n \int_{K/M}\int_{M/Z(G)} \left\langle w_i,\chi_\tau(km)f(km) \right\rangle dmdkM\cdot w_i\\
=&\sum_{i=1}^n \langle w_i,f(e)\rangle\cdot w_i=f(e).
\end{align*}
Since $W$ occurs as a subrepresentations of $V_\tau$, one has $n=\dim_\C W\leq d_\tau$. In particular, $$\|f_i\|\leq \|f_i\|_\infty\leq d_\tau^2\qand
\|f(k)\|\leq \sum_{i=1}^n \|\pi_\sigma(k)f_i\|\|w_i\| \|f\|\leq d_\tau^3 \|\tau(k)\|\|f\|$$
combined the previous proposition imply the assertion with an appropriate choice for $|\cdot|$.
\end{proof}
Let $|\nu|:=\sup_{H\in \Omega_\a}|\nu(H)|$, for $\nu \in \a^\ast$.
\begin{thm}\label{thmPrinc}
Let $G$ be a connected semisimple Lie group. There exist $C_0,N_1,N_2>0$ with the following property. For every irreducible unitary representation $(\sigma_0,W)$ of $M$, 
$$\mu\in \a_\C^\ast\qand\sigma:=\sigma_0\times \mu\times 1:P_{\rm min}\to \GL(W),\quad man\mapsto \sigma_0(m)a^\mu$$
as well as $[\tau]\in \widehat{K}$, then one has
$$\|e^{i\del\pi_\sigma(x)}v\|_{\cH_\sigma}\leq C_0\cdot d_\tau^3\cdot e^{|\Im(\mu)|}\cdot \left(\pitwo-r_{\rm spec}(\ad(x))\right)^{-(N_1\cdot |\Re(\mu)+\rho|+N_2\cdot |\tau|)},
$$
for all $x\in \Omega_\p$ and and $v\in \cH_{\sigma,[\tau]}$ with $\|v\|_{\cH_\sigma}=1$.
\end{thm}
\begin{proof}
We may w.l.o.g. assume that $x\in \Omega_\a$. It follows from $\mu_{K/M}(K/M)=1$ that
\begin{align*}
\|e^{i\del\pi_\sigma(x)}v\|_{\cH_\sigma}^2=&\int_{K/M}\left|\alpha(\exp(-ix)k)^{-(\mu+\rho)}\right|^2\|v(\kappa(\exp(-ix)k))\|^2dkM\\
\leq &\sup_{k\in K}\left|\alpha(\exp(-ix)k)^{-(\mu+\rho)}\right|^2\cdot \left(\sup_{k\in K} \|v(\kappa(\exp(-ix)k))\|^2\right).
\end{align*}
Corollary \ref{corKgrowth} and Theorem \ref{ThmMainResult} imply that, for $C_0,N_2>0$ sufficiently large, one has
\begin{align*}
\sup_{k\in K} \|v(\kappa(\exp(-ix)k))\|\leq C_0\cdot d_\tau^3\cdot  \left(\pitwo-r_{\rm spec}(\ad(x))\right)^{-N_2\cdot |\tau|}.
\end{align*}
The Complex Convexity Theorem (cf. \cite[Thm. 3.1]{GK02b}) implies that
$$\Im (H(\exp(-ix)k))\in -\cW_\a.x\subseteq \Omega_\a,\qforall x\in \Omega_\a\qand k\in K.
$$
In particular, $|\alpha(\exp(-ix)k)^{-i\Im(\mu)}|^2\leq e^{|\Im(\mu)|}$ follows. Hence, Theorem \ref{thmasymp} implies that
$$\sup_{k\in K}\left|\alpha(\exp(-ix)k)^{-(\nu+\rho)}\right|^2\leq C_1\left(\pitwo-r_{\rm spec}(\ad(x))\right)^{-N_1|\nu+\rho|},\qforall \nu \in \a^\ast
$$
and $C_1,N_1>0$ sufficiently large. Therefore, the assertion follows from 
\begin{align*}
\left|\alpha(\exp(-ix)k)^{-(\mu+\rho)}\right|^2\leq e^{|\Im(\mu)|} \left|\alpha(\exp(-ix)k)^{-(\Re(\mu)+\rho)}\right|^2.\tag*{\qedhere}
\end{align*}
\end{proof}
\section{Distributional limits at the boundary of the crown}
In Subsection \ref{subsecSlow}, we discuss polynomial growth rates of analytic vectors in locally convex vector spaces and the existence of the corresponding limits in the space of distribution vectors. Afterward, in Subsection \ref{subsTemp}, we prove that the polynomial growth rates for principal series representations can be quantitatively transferred to irreducible unitary representations. We use these results to prove that the Kr\"otz--Stanton Extension Theorem holds for reducible unitary representations on a dense subspace of the $K$-finite vectors on which one also has polynomial growth rates at the boundary.
\subsection{Analytic vectors of slow growth}\label{subsecSlow}
In this subsection, we consider strongly continuous one-parameter groups on topological vector spaces and holomorphic extensions of orbit maps of weakly analytic vectors. First, we discuss the special case of Banach representations, which is the motivation for this section. Fix a Banach space $X$ and a strongly continuous representation $\pi:\R\to B(X)^\times$. The following definition is motivated by the definition for holomorphic functions on CR manifolds that grow at most polynomially as they approach an edge (cf. \cite[Section \S 7.2]{BER99}). We will elaborate on this further in the following subsection.
\begin{defin}
\rm{Let $(\pi,X)$ be a Banach representation of $\R$. Suppose $v\in X$ is an analytic vector of $\pi$, for which the orbit map $$\pi^v:\R\to X,\quad t\mapsto \pi(t)v$$ extends to a holomorphic map on the strip $\cS_{\pm \beta}$, for $\beta>0$. We denote this set of such vectors by $X_\beta^\omega$. We say that $v\in X_\beta^\omega$ has \textit{slow growth} on $\cS_{\pm \beta}$ if there exist $C> 0$ and $N\geq 0$ such that 
$$\|\pi^v(it)\|\leq C\cdot (\beta-|t|)^{-N},\qforall |t|<\beta.
$$
In this case, we say that a vector $v$ of slow growth \textit{is of order} $\leq N$.
}
\end{defin}
\begin{rem}
\rm{For $N^\prime\geq N$, one has that $(\beta-|t|)^{N^\prime-N}$ is bounded, for $|t|<\beta$. Therefore, for any vector $v\in X$ of slow growth with order $N$ and fixed $N^\prime\geq N$, there exists $C^\prime>0$ such that
$$\|\pi^v(it)\|\leq C^\prime\cdot (\beta-|t|)^{-N^\prime},\qforall |t|<\beta.
$$
}
\end{rem}
\begin{lem}\label{lemdiff}
If $v\in X$ is a vector of slow growth of order $\leq N$. Then $\left.\tfrac{d^k}{dx^k}\right\vert_{x=0}\pi(x)v$ is a vector of slow growth of order $\leq N+k$.
\end{lem}
\begin{proof}
The Cauchy-Riemann equations imply that

$$\left\|\left.\frac{d^k}{dx^k}\right\vert_{x=0}\pi^v(x+it)\right\|=\left\|\frac{d^k}{dt^k}\pi^v(it)\right\|.
$$
For $\eta\in X^\prime$, the function 
$$f_\eta:\cS_{\pm \beta}\to \C,\quad x+it\mapsto \eta(\pi^v(x+it))$$
is holomorphic and thus, for $0\leq t<\beta$, $0<\varepsilon<\beta-t-\varepsilon$ and $r_\varepsilon:=\beta-t-\varepsilon$, one obtains
\begin{align*}
\left| \frac{d^k}{dt^k}\eta\left(\pi^v(it)\right)\right|= \left|\frac{k!}{2\pi i}\int_{|z-it|=r_\varepsilon}\frac{f_\eta(z)}{(z-it)^{k+1}}dz\right|\leq \frac{C\cdot\|\eta\|\cdot k!}{(\beta-t-\varepsilon)^{k}} (\beta-t)^{-N}.
\end{align*}
Letting $\varepsilon$ approach zero and using similar arguments, the assertion follows from
\begin{align*}
\left\|\frac{d^k}{dt^k}\pi^v(it)\right\|\leq C\cdot  k!\cdot (\beta-|t|)^{-(N+k)},\qforall |t|<\beta.\tag*{\qedhere}
\end{align*}
\end{proof}
\begin{defin}
\rm{Let $V$ be a locally convex topological vector space over $\C$ and let $(p_j)_{j\in J}$ be a family of continuous seminorms on $V$ defining the topology. For a strongly continuous one-parameter group $\pi:\R\to \GL(V)$, we define on the space of smooth vectors $V^\infty$ the locally convex topology defined by the seminorms
$$p_{j,n}:V^\infty\to [0,\infty),\quad v\mapsto p_j(\pi(\del_x)^nv),\qfor j\in J \qand n\in \N_0.
$$
We say that an analytic vector $v\in V$ has \textit{slow growth} on $\cS_{\pm \beta}$, if $\pi^v:\R\to V$ extends to a holomorphic map $\pi^v:\cS_{\pm \beta}\to V$ such that for all $j\in J$, $n\in \N_0$ there exist $C>0$ and $N\geq 0$ such that
$$p_{j,n}(\pi^v(it))\leq C\cdot (\beta-|t|)^{-N},\qforall |t|<\beta.
$$
}
\end{defin}
\begin{prop}
Let $\pi:\R \to B(X)^\times$ be a strongly continuous Banach representation and $\beta>0$. Then $X_\beta^\omega\subseteq X^\infty$ and the concepts of slow growth for $X$ and $X^\infty$ coincide.
\end{prop}
\begin{proof}
This follows immediately from Lemma \ref{lemdiff}.
\end{proof}
\begin{rem}
\rm{We now discuss the generalization to arbitrary topological vector spaces $V$. Recall that a subset $B\subseteq V$ is called bounded if and only if for all zero neighbourhoods $U$ in $V$ there exists $r>0$ such that $B\subseteq rU$. We denote by $\cB$ the collection of all bounded sets and define on the continuous dual $V^\prime$ the strong topology as the locally convex topology induced by the seminorms
$$|\lambda|_B:=\sup_{v\in B}|\lambda(v)|,\qfor B\in \cB\qand \lambda\in V^\prime.
$$
We denote $V^\prime$ endowed with the strong topology by $V_{\rm str}^\prime$ and consider the space of smooth vectors $(V_{\rm str}^\prime)^\infty$ for the dual representation
$$\pi^\ast:\R\to \GL(V^\prime),\quad  (\pi^\ast(t)\lambda)(v):=\lambda(\pi(-t)v),\qfor \lambda \in V^\prime\qand v\in V.
$$
On $(V_{\rm str}^\prime)^\infty$, we define the family of seminorms
$$|\lambda|_{B,n}:=|\pi^\ast(\del_x)^n\lambda|_B,\qfor \lambda\in (V_{\rm str}^\prime)^\infty, n\in \N_0\qand B\in \cB
$$
and denote with $(V_{\rm str}^\prime)^{-\infty}$ the continuous dual space of $(V_{\rm str}^\prime)^\infty$. Note that the evaluation maps $\ev_x:V\to (V_{\rm str}^\prime)^{-\infty}$ are continuous linear maps since $|\text{ev}_x(\lambda)|=|\lambda|_{\{x\}}$ and $\{x\}\in \cB$. 
}
\end{rem}
\begin{rem}
\rm{If $(\pi,\cH)$ is a Hilbert representation of $\R$, then by the Riesz-Fr\'{e}chet representation theorem $$\Phi:\cH\to \cH^\prime,\quad v \mapsto \langle v,\cdot \rangle$$ is an antilinear isomorphism. A subset $B\subseteq \cH$ is bounded if and only if there exists $r>0$ such that $B\subseteq B_r(0)$. Thus, the strong topology on $\cH^\prime$ coincides with the topology induced by $\Phi$ and $(\cH_{\rm str}^\prime)^\infty$ is identified with the smooth vectors $\cH^\infty$, which consists of $v\in \cH$ for which the orbit map $\pi^v:\R\to \cH$ is smooth and carries the topology induced by the seminorms
$$p_n:\cH^\infty\to [0,\infty),\quad p_n(v):=\left\|d\pi(\del_x)^n v\right\|.
$$
In particular, $(\cH_{\rm str}^\prime)^{-\infty}$ is identified with the space of distribution vectors $\cH^{-\infty}$.
}
\end{rem}
\begin{defin}
\rm{Let $\pi:\R\to \GL(V)$ be a strongly continuous representation on a complex topological vector space $V$. 
\begin{itemize}
\item[(i)] We call a vector $v\in V$ a \textit{weakly analytic vector}, if there exists $\varepsilon>0$ such that the orbit map $\pi^v:\R\to V,\; t\mapsto \pi(t)v$ extends to a weakly holomorphic map $\cS_{\pm \varepsilon}\to V$. We denote the space of weakly analytic vectors with orbit maps that extend to a holomorphic map on $\cS_{\pm \beta}$ by $V_\beta^\omega$.
\item[(ii)] We say that a weakly analytic vector $v\in V_\beta^\omega$ has \textit{slow growth} on $\cS_{\pm \beta}$ if there exist $N\geq 0$ and a bounded subset $B\subseteq V$ such that
$$(\beta-|t|)^N\pi^v(it)\in B,\qforall |t|<\beta.
$$
\end{itemize}
}
\end{defin}
\begin{rem}
\rm{Note that if $v\in V_\beta^\omega$ and $B\subseteq V$ is a bounded subset as above, then $\pi^v:\cS_{\pm \beta}\to V$ is locally bounded since $\R$ is locally compact and $\pi(C)B$ is bounded, for $C\subseteq \R$ compact. In particular, if $z\in \cS_{\pm \beta}$ and $|\Im(z)|<\varepsilon<\beta$, then $$\pi^v(C+i(-\varepsilon,\varepsilon))\subseteq (\beta-\varepsilon)^{-N}\pi(C)B.$$
In particular, if $V$ is a sequentially complete locally convex vector space, then a weakly analytic vector of slow growth is analytic (cf. \cite[Prop. A.III.2]{Ne00}). We choose to consider weakly analytic vectors as weak analyticity is what is essential to us and to avoid specifying different notions of holomorphy. Furthermore, since $V^\prime$ separates the points, $V\hookto (V_{\rm str}^\prime)^{-\infty}$ is an embedding.
}
\end{rem}
\begin{lem}\label{lemsmoothext}
{\rm (cf. \cite[Lemma 7.2.12]{BER99})} Let $n\in \N_0$ and $F:X\times[0,\beta)\to \C$ be a function such that 
$$F_x:[0,\beta)\to \C,\quad t\mapsto F(x,t)$$
is a $C^{n+1}$-function, for all $x\in X$. If there exists $C:X\to (0,\infty)$ such that 
$$|F_x^{(n+1)}(t)|\leq C(x)\cdot(\beta-t)^{-n},\qforall x\in X,
$$
then $F_x$ extends to a continuous function $F_x:[0,\beta]\to \C$, for every $x\in X$, and there exists $c_0>0$ such that
$$|F_x(\beta)|\leq c_0\cdot C(x),\qforall x\in X.
$$
\end{lem}
\begin{proof}
Multiple applications of the fundamental theorem of calculus yield
$$F_x(t_0)-\sum_{k=0}^{n+1}\frac{F_x^{(k)}(0)}{k!}t_0^k=\int_0^{t_0}dt_1\dots \int_0^{t_{n}}dt_{n+1}F_x^{(n+1)}(t_{n+1}).
$$
For $t_0\leq t_0^\prime<\beta$, this implies that
\begin{align*}
|F_x(t_0)-F_x(t_0^\prime)|\leq &\int_{t_0}^{t_0^\prime}dt_1\dots \int_0^{t_{n}}dt_{n+1}|F_x^{(n+1)}(t_{n+1})|\\
\leq& \int_{t_0}^{t_0^\prime}dt_1\underbrace{\int_0^{t_{1}}dt_{2} \dots \int_0^{t_{n}}dt_{n+1} \frac{C(x)}{(\beta-t_{n+1})^n}}_{=:\,C(x)\cdot g(t_1)}.
\end{align*}
Since the function $[0,\beta)\ni t\mapsto g(t)$ is given by a multiple of $p(t)-\ln(\beta-t)$ which does not depend on $x\in X$, for some polynomial $p$, it is integrable and thus one has $$\lim_{t_0\to \beta}|F_x(t_0)-F_x(t_0^\prime)|=0,\qforall x\in X.$$
Hence $F_x$ extends to a continuous function $F_x:[0,\beta]\to \C$. Furthermore, the assertion follows from the estimate
\begin{align*}
|F_x(t_0)-F_x(t_0^\prime)|\leq c_0\cdot C(x),\qfor c_0:= \int_0^{\beta} g(t)dt.\tag*{\qedhere}
\end{align*}
\end{proof}
We can now prove an analogue of \cite[Thm. 7.2.6]{BER99}, which states that the boundary values of holomorphic functions of slow growth exist in the space of distributions. The following theorem generalizes \cite[Thm. 2.26]{FNO23} on unitary representations of $\R$.
\begin{thm}\label{distrlimit}
Let {\rm~$\pi:\R\to \GL(V)$} be a strongly continuous representation on a complex topological vector space $V$ and $v\in V_\beta^\omega$ be an analytic vector with slow growth on $\cS_{\pm \beta}$. Then $$\lim_{t\nearrow \beta}\pi^v(it)\in(V_{\rm str}^\prime)^{-\infty}$$ exists in the weak-$\ast$-topology.
\end{thm}
\begin{proof}
Let $B\subseteq V$ be a bounded subset and $N\in \N$ such that
$$(\beta-|t|)^N\pi^v(it)\in B,\qfor |t|<\beta.
$$
For $\eta\in (V_{\rm str}^\prime)^\infty$, we consider the holomorphic function $$F_\eta:\cS_\beta\to \C,\quad x+it\mapsto\eta(\pi^v(x+it)).$$ The Cauchy-Riemann equations imply that, for $n\in \N_0$, one has
\begin{align*}
|(\del_t^n F_\eta)(it)|=&|(\del_x^n F_\eta)(it)|=|\eta((\del_x^n\pi^v)(it)|\leq |\eta|_{B,n}\cdot (\beta-|t|)^{-N},
\end{align*}
for all $|t|<\beta$. In particular, Lemma \ref{lemsmoothext} implies that $\lim_{t\to \beta}F_\eta(it)$ exists. The assertion now follows, since there exists $c_0>0$ such that
\begin{align*}
\left|\lim_{t\to \beta}\eta(\pi^v(it))\right|=\left|\lim_{t\to \beta}F_\eta(it)\right|\leq c_0\cdot |\eta|_{B,N+1},\qforall \eta\in (X^\prime)^\infty.\tag*{\qedhere}
\end{align*}
\end{proof}
\begin{cor}\label{corSlowGrowth}
Let $(\pi,\cH)$ be a strongly continuous Hilbert representation of $\R$ and $v\in \cH_\beta^\omega$ be a vector of slow growth. Then $$\lim_{t\nearrow \beta}\pi^v(it)\in \cH^{-\infty}$$
exists in the weak-$\ast$-topology.
\end{cor}

\subsection{Tempered vectors in unitary representations}\label{subsTemp}
Let $G$ be a connected semisimple Lie group. For  a unitary representation $(\pi,\cH)$, let
$$\cH_{\Omega_\p}:=\bigcap_{x\in \Omega_\p}\cD(e^{i\del\pi(x)}).
$$
If $\pi$ is irreducible, then $\cH^{[K]}\subseteq \cH_{\Omega_\p}$ is a dense subspace in $\cH$ (cf. Theorem \ref{thmKrSt}). The main results of this subsection are that for an arbitrary unitary representation $(\pi,\cH)$ with $Z(G)$-character, the space of \textit{tempered vectors} $\cH_{\rm temp}$ defined as
$$\left\{v\in \cH_{\Omega_\p}\;|\;\exists C,N>0 \;:\;\|e^{i\del\pi(x)}v\|\leq C\left(\pitwo-r_{\rm spec}(\ad(x))\right)^{-N}\qforall x\in \Omega_\p\right\}
$$
is dense in $\cH$ and if $\pi$ is irreducible, every $K$-finite vector is tempered. 
\begin{rem}
\rm{Let $(\pi,\cH)$ be an irreducible unitary representation of a connected semisimple Lie group $G$. In \cite[Appendix A.1]{FNO23}, the authors argue that the Casselman Subrepresentation Theorem (cf. \cite[Prop. 8.23]{CM82}) generalizes to arbitrary connected semisimple Lie groups, i.e. there exists an embedding of $(\g,K)$-modules
$$\varphi:\cH^{[K]}\hookto \cH_{\sigma,\lambda}^{[K]}:=\cH_{\sigma\times \lambda\times 1}^{[K]},
$$
for an irreducible unitary representation $[\sigma]\in \widehat{M}$, $\lambda\in \a_\C^\ast$ and the corresponding principal series representation associated to the irreducible representation $\sigma\times \lambda\times 1$ of $P_{\rm min}=MAN$. The functionals $\lambda+\rho\in \a_\C^\ast$, for which there exists an embedding $\cH^{[K]}\hookto \cH_{\sigma,\lambda}^{[K]}$ are called \textit{leading characters} of $(\pi,\cH)$. Since the representation $(\pi,\cH)$ is unitary, it has bounded matrix coefficients $$\pi^{v,w}:G\to \C,\quad g\mapsto\langle v,\pi(g)w\rangle,\qfor v,w\in \cH.$$
Then, $\pi^{v,w}$ is a $K$-left and $K$-right finite $Z(\g)$-finite smooth function on $G$. In \cite{CM82} such functions are studied under the assumption that $G$ has finite center. This assumption is never used since the algebraic prerequisites of their theory only require that $K$ operates on $\g$ as a compact group. In particular, by defining radial components in the same manner and using $G=KAK$, one arrives at the same analytic differential equation with regular singularities for the restriction of $\pi^{v,w}$ to $A^+=\exp(\a^+)$, where $\a^+\subseteq \a$ is a Weyl chamber. Therefore, all results immediately generalize. In particular, \cite[Thm. 8.11]{CM82} and the fact that $\pi$ has bounded matrix coefficients implies that 
$$\Re(\lambda)(H)\leq \rho(H),\qforall H\in \overline{\a_+}.$$
}
\end{rem}
\begin{lem}
Let $(\pi,\cH)$ be an irreducible unitary representation of a connected semisimple Lie group $G$ and let 
$$\varphi:\cH^{[K]}\hookto \cH_{\sigma,\lambda}^{[K]}$$
be an embedding of $(\g,K)$-modules. Let $V:=\im(\varphi)$ and $\psi:=(\varphi^{-1})^\ast: \cH^{[K]}\to V$. Then, one has
\begin{equation}\label{EqScp}
\langle \varphi(v),\psi(w)\rangle_{\cH_{\sigma,\lambda}}=\langle v,w\rangle_{\cH},\qforall v,w\in \cH^{[K]}
\end{equation}
and $\psi$ is a $(\g,K)$-module for the $(\g,K)$-module structure on $\cH_{\sigma,\lambda}$ induced by the dual representation $\pi_{\sigma,\lambda}^\ast(g):=\pi_{\sigma,\lambda}(g^{-1})^\ast$.
\end{lem}
\begin{proof}
As $\varphi:\cH^{[K]}\to V$ is an isomorphism, its inverse $\varphi^{-1}:V\to \cH^{[K]}$ maps the finite-dimensional $K$-isotypical components to finite-dimensional $K$-isotypical components. Hence, the adjoint $\psi=(\varphi^{-1})^\ast$ can be defined on the $K$-types and be extended by linearity. In particular, $\psi$ clearly satisfies (\ref{EqScp}). It follows from
\begin{align*}
\langle\varphi(v),\pi_{\sigma,\lambda}(k)^{-1}\psi(w)\rangle_{\cH_{\sigma,\lambda}}=&\langle\pi_{\sigma,\lambda}(k)\varphi(v),\psi(w)\rangle_{\cH_{\sigma,\lambda}}=\langle\varphi(\pi(k)v),\psi(w)\rangle_{\cH_{\sigma,\lambda}}\\
=&\langle \pi(k)v,w\rangle_\cH
=\langle v,\pi(k)^{-1}w\rangle_\cH\\
=&\langle\varphi(v),\psi(\pi(k)^{-1}w)\rangle_{\cH_{\sigma,\lambda}}
\end{align*}
that $\psi(\pi(k)w)=\pi_{\sigma,\lambda}(k)\psi(w)$, i.e. $\psi$ is a $K$-module homomorphism. Similarly
\begin{align*}
\langle\varphi(v),d\pi_{\sigma,\lambda}(X)^\ast\psi(w)\rangle_{\cH_{\sigma,\lambda}}=&\langle d\pi_{\sigma,\lambda}(X)\varphi(v),\psi(w)\rangle_{\cH_{\sigma,\lambda}}=\langle\varphi(d\pi(X)v),\psi(w)\rangle_{\cH_{\sigma,\lambda}}\\
=&\langle d\pi(X)v,w\rangle_\cH
=\langle v,d\pi(X)^\ast w\rangle_\cH\\
=&\langle\varphi(v),\psi(d\pi(X)^\ast w)\rangle_{\cH_{\sigma,\lambda}}
\end{align*}
and $d\pi(X)^\ast=-d\pi(X)$ imply that $\psi(d\pi(X)w)=-d\pi_{\sigma,\lambda}(X)^\ast\psi(w)=d\pi_{\sigma,\lambda}^\ast(X)\psi(w)$.
\end{proof}
\begin{thm}\label{thmUnitGrowth}
Let $G$ be a connected semisimple Lie group and $(\pi,\cH)$ be an irreducible unitary representation.
\begin{itemize}
\item[{\rm (i)}] 
For a leading character $\lambda\in \a_\C^\ast$ of $(\pi,\cH)$ and $\sigma\in \widehat{M}$ such that there exists a non-zero $(\g,K)$-module homomorphism $\varphi:\cH^{[K]}\to\cH_{\sigma,\lambda}^{[K]}$
one has
$$\|e^{i\del\pi(x)}v\|_\cH^2=\left\langle e^{i\del\pi_{\sigma,\lambda}(x)}\varphi(v),e^{i\del\pi_{\sigma,\lambda}^\ast(x)}\psi(v)\right\rangle_{\cH_{\sigma,\lambda}},$$
for all $x\in \Omega_\p$ and $v\in \cH^{[K]}$.
\item[{\rm (ii)}] Every $K$-finite vector $v\in \cH^{[K]}$ is tempered and there exist $N_1,N_2>0$ independent of $(\pi,\cH)$ and $v\in \cH^{[K]}$ and $C_{v,\pi}>0$ such that
$$\|e^{i\del\pi(x)}v\|^2\leq C_{v,\pi}\cdot \left(\pitwo-r_{\rm spec}(\ad(x))\right)^{-(N_1+N_2|\tau|)},
$$
for all $[\tau]\in \widehat{K}$, $v\in \cH_{[\tau]}^{[K]}$ and $x\in \Omega_\p$.
\end{itemize}

\end{thm}
\begin{proof}
(i) Consider the analytic kernels $K_i:G\times G\to \C$ defined as
$$K_1(g_1,g_2):=\langle \pi(g_1)v,\pi(g_2)v\rangle_\cH\qand K_2(g_1,g_2):=\langle \pi_{\sigma,\lambda}(g_1)\varphi(v),\pi_{\sigma,\lambda}^\ast(g_2)\psi(v)\rangle_{\cH_{\sigma,\lambda}}.
$$
For $D_1,D_2\in \cU(\g)$, one has
\begin{align*}
R_{D_1,D_2}K_1(e,e)=&\langle d\pi(D_1)v,d\pi(D_2)v\rangle_\cH=\langle \varphi (d\pi(D_1)v),\psi(d\pi(D_2)v)\rangle_{\cH_{\sigma,\lambda}}\\
=&\langle d\pi_{\sigma,\lambda}(D_1)\varphi(v),\pi_{\sigma,\lambda}^\ast(D_2)\psi(v)\rangle_{\cH_{\sigma,\lambda}}=R_{D_1,D_2}K_2(e,e).
\end{align*}
Therefore, $K_1,K_2:G\times G\to \C$ are analytic functions on a connected analytic manifold whose derivatives coincide in $(e,e)$, i.e. $K_1=K_2$ follows. The kernel $K:=K_1=K_2$ extends to a holomorphic function $K:\wxi\times \wxi\to \C$ by Proposition \ref{PropMatrixCoeff}, which satisfies
\begin{align*}
\left\|e^{i\del\pi(x)}v\right\|_\cH^2=K(\exp(-ix),\exp(ix))=\left\langle e^{i\del\pi_{\sigma,\lambda}(x)}\varphi(v),e^{i\del\pi_{\sigma,\lambda}^\ast(x)}\psi(v)\right\rangle_{\cH_{\sigma,\lambda}}.
\end{align*}
(ii) The Cauchy-Schwarz Inequality and (i) imply that
$$\left\|e^{i\del\pi(x)}v\right\|_\cH^2\leq \left\| e^{i\del\pi_{\sigma,\lambda}(x)}\varphi(v)\right\|_{\cH_{\sigma,\lambda}}\left\|e^{i\del\pi_{\sigma,\lambda}^\ast(x)}\psi(v)\right\|_{\cH_{\sigma,\lambda}}.
$$
Since $(\pi_{\sigma,\lambda}^\ast,\cH_{\sigma,\lambda})\cong (\pi_{\sigma^\ast,-\overline{\lambda}},\cH_{\sigma^\ast,-\overline{\lambda}})$, the assertion follows from Theorem \ref{thmPrinc}.
\end{proof}
\begin{rem}
\rm{By Theorem \ref{thmPrinc}, the constant $C_{v,\pi}$ in the theorem above is given by  $$C_{v,\pi}:=C\cdot d_\tau^3\cdot e^{|\Im(\lambda)|}\cdot\|\varphi(v)\|_{\cH_{\sigma,\mu}}\|\psi(v)\|_{\cH_{\sigma,\mu}}.$$
For our purposes, its existence is sufficient. What is more surprising to us and useful for our purposes, is that the exponent $N_1+N_2|\tau|$ only depends on $G$ and the $K$-type $[\tau]\in \widehat{K}$.
}
\end{rem}
\begin{thm}
Let $(\pi,\cH)$ be a unitary representation of a connected semisimple Lie group $G$ with $Z(G)$-character. Then the space $\cH_{\rm temp}^{[K]}:=\cH_\temp\cap \cH^{[K]}$ is dense in $\cH$.
\end{thm}
\begin{proof}
Consider the central subgroup $$Z_\chi:=\{(z,\chi(z)^{-1})\;|\;z\in Z(G)\}\subseteq K\times \T$$
and define $K\times_\chi\T:=(K\times \T)/Z_\chi$. Since $\Ad(K)\cong K/Z(G)$ is compact, the group $K_\chi$ is compact as a bundle over a compact manifold with compact fibres. The representation $\pi\vert_K$ extends to a representation of $K_\chi$ by
$$\pi([(k,z)]):=z\pi(k), \qfor k\in K\qand z\in \T.
$$
Since $\pi(K)$ and $\pi(K_\chi)$ leave the same subspaces invariant, it follows from Peter-Weyl's Theorem applied to the compact group $K_\chi$ that $\cH$ decomposes into finite-dimensional $K$-invariant subspaces and $ \cH^{[K]}\subseteq \cH$ is dense. It therefore suffices to prove that every $K$-finite vector can be approximated by $K$-finite tempered vectors. The desintegration theory of representations of $G$ and the separable liminal C$^\ast$-algebra $C^\ast(G)$ are the same. Therefore, \cite[Thm. 8.6.5]{Dix69} implies that there exist a measure space $(X,\Sigma,\mu)$ and a measurable field of irreducible unitary representations $(\pi_x,\cH_x)_{x\in X}$ of $G$ with $Z(G)$-character $\chi$ such that $$(\pi,\cH)\cong\int_X^\oplus(\pi_x, \cH_x)d\mu(x).$$
We assume from now on $\cH=\int_X^\oplus\cH_xd\mu(x)$. If $\cE\subseteq \cH^{[K]}$ is a finite-dimensional $K$-invariant subspace, then the inclusion $\cE\hookto \cH$ is Hilbert-Schmidt and $K$-equivariant. In particular, the inclusion is pointwise defined, which means that there exist uniquely determined $K$-equivariant continuous linear maps $T_{x,\cE}:\cE\to \cH_x$ such that $v=(T_{x,\cE}(v))$, for all $v\in \cE$ (cf. \cite[Lemma 1.3, Thm. 1.5]{Be88}). By decomposing $\cH^{[K]}$ into finite-dimensional $K$-invariant subspaces, one obtains $(\g,K)$-module homomorphisms $T_x:\cH^{[K]}\to \cH_x^{[K]}$ such that $v=(T_x(v))_{x\in X}$, for all $v\in \cH^{[K]}$. For $[\tau]\in \widehat{K}$ and $v\in \cH_{[\tau]}^{[K]}$, let $v_x:=T_x(v)$ and consider vector field
$$e^{i\del\pi(h)}v:X\to \coprod_{x\in X} \cH_x,\quad x\mapsto e^{i\del\pi_x(h)}v_x=\sum_{k=0}^\infty \frac{i^k}{k!}d\pi_x(h)^kv_x
$$
is well-defined for all $h\in \Omega_\p$ and measurable as a series of the measurable vector fields $d\pi(h)^kv$, which are measurable as derivatives of measurable vector fields. Since $v_x\in \cH_{x,[\tau]}^{[K]}$, for $\mu$-almost all $x\in X$, Theorem \ref{thmUnitGrowth} implies that there exists $N_1,N_2>0$ such that
$$C_x:=\sup_{h\in \Omega_\p} \frac{\|e^{i\del\pi_x(h)}v_x\|}{\|v_x\|}\left(\pitwo-r_{\rm spec}(\ad(h))\right)^{N_1+N_2|\tau|}<\infty,\qforall x\in X.
$$
Note that the map
$$\Omega_\p\ni h\mapsto \frac{\|e^{i\del\pi_x(h)}v_x\|}{\|v_x\|}\left(\pitwo-r_{\rm spec}(\ad(h))\right)^{N_1+N_2|\tau|}$$
is continuous and there exists a countable dense subset of $\Omega_\p$. In particular, the map $C:X\to (0,\infty),\;x\mapsto C_x$ is measurable as a supremum over a countable subset and
$$\|e^{i\del\pi_x(h)}v_x\|\leq C_x\cdot \left(\pitwo-\rho(h)\right)^{-(N_1+N_2|\tau|)}\|v_x\|,\qforall h\in \Omega_\p \qand x\in X.
$$
For $X_n:=C^{-1}((0,n))\subseteq X$ and the characteristic function $\mathds{1}_{X_n}$, the vector fields $$v_n:X\to \coprod_{x\in X}\cH_x,\quad x\mapsto \mathds{1}_{X_n}(x)\cdot v_x$$ 
are measurable, square-integrable and $K$-finite. Furthermore,
$$\|e^{i\del\pi(h)}v_n\|^2=\int_{X_n} \|e^{i\del\pi_x(h)}v_x\|^2d\mu(x)\leq n^2\cdot \left(\pitwo-\rho(h)\right)^{-2(N_1+N_2|\tau|)}\cdot \|v\|^2<\infty
$$
implies that $v_n\in \cH_\temp^{[K]}$ and thus the assertion follows from $v_n\to v$ as $n\to \infty$.
\end{proof}

\newpage
\appendix
\section{Holomorphic extension of $(Z(\g),K)$-finite functions}\label{subsDiffEq}
\setcounter{subsection}{1}
In this section, we argue that a similar theorem as the following holds for arbitrary connected semisimple Lie groups and $(Z(\g),K)$-right finite smooth functions $f\in C^\infty(G)$.
\begin{thm*}{\rm \cite[Thm. 3.5]{KSch09}}\\
Let $G\subseteq G_\C$ be a linear Lie group and $f\in C^\infty(G)$ be a $K$-right finite eigenfunction of the quadratic Casimir operator $C_\g$. Then, $f$ extends to a holomorphic function on 
$$\Xi_{G_\C}=G\exp(i\Omega_\p)K_\C\subseteq G_\C.
$$
\end{thm*}
The essential ideas from \cite{KSch09} immediately generalize to arbitrary connected Lie groups $G$, with our primary contribution being the replacement of the condition on $f$ being a $C_\g$-eigenfunction by the condition on $f$ being $C_\g$-finite. This does not require any new methods as the framework, we work with, has already been extensively developed in \cite{KS05,KSch09} and a differential operator and a polynomial in this differential operator have the same characteristic variety (cf. Definition \ref{defDiffOp}(iii)). 
\begin{defin}\label{defDiffOp}
\rm{For $i=1,2$, let $q_i:X_i\epi M$ be smooth complex vector bundles with typical fibre $V_i$ over a fixed smooth manifold $M$ and let $\Gamma^\infty(M,X_i)$ be the corresponding spaces of smooth sections. 
\begin{itemize}
\item[(i)] A linear operator $P:\Gamma^\infty(M,X_1)\to \Gamma^\infty(M,X_2)$ is called a \textit{linear differential operator} if $P$ does not increase the support of sections, which means that, in local coordinates, $P$ is given by a linear differential operator
$$C_c^\infty(U,V_1)\to C_c^\infty(U,V_2),
$$
for local trivializations $(\varphi_i,U)$ of $q_i:X_i\to M$. For $x\in M$ and a chart $(\varphi,U)$ with coordinates $(x_1,\dots,x_n)$, the operator $P$ is given in these coordinates by
$$\sum_{|\alpha|\leq N} P_\alpha(x)\frac{\del^{\alpha_1}}{\del x_1^{\alpha_1}}\dots \frac{\del^{\alpha_n}}{\del x_1^{\alpha_n}},\qfor P_\alpha\in C^\infty(U, \Hom_\C(V_1,V_2)).
$$
\item[(ii)] The \textit{principal symbol} is defined as the map
$$\sigma(P):T^\ast M\to\Hom_\C(V_1,V_2),\quad \sigma(P)(\gamma):=\sum_{|\alpha|=m}P_\alpha(x)\xi^\alpha$$ 
for $\gamma=\sum_{j=1}^n \xi_jdx_j(x)\in T_xM^\ast$ and $\xi^\alpha:=\xi_1^{\alpha_1}\dots\xi_n^{\alpha_n}$. This definition is independent of the chosen chart (cf. \cite[Section 6.4]{H\"or89}). A linear differential operator is called \textit{elliptic} if its principal symbol is invertible, for all $0\neq \gamma\in T^\ast M$.
\item[(iii)] The \textit{characteristic variety} $\text{char}(P)$ of $P:\Gamma^\infty(M,X_1)\to \Gamma^\infty(M,X_2)$ is defined as
$$\{\gamma\in T^\ast M\;|\;\sigma(P)(\gamma)\in \Hom_\C(V_1,V_2)\quad\text{is not invertible}\}.$$
\end{itemize}
}
\end{defin}
\begin{rem}\label{remAnalyticReg}
\rm{Suppose $P\in \End(\Gamma^\infty(M,X))$ is an elliptic linear differential operator on an analytic complex vector bundle $X\epi M$ with typical fibre $V$. Then, the characteristic variety of $P$ is empty. Suppose $f\in \Gamma^\infty(M,X)$ satisfies the differential equation $Pf=0$. Then \cite[Thm. 8.6.1]{H\"or89} implies that, for the analytic wave front set $\text{WF}_A(f)$ of $f$, one has 
$$\text{WF}_A(f)\subseteq \text{char}(P)\cup \text{WF}_A(Pf)=\emptyset$$
and thus $f$ is analytic since its analytic wave front set is empty. We refer to this as the \textit{Analytic Elliptic Regularity Theorem}. We refer to \cite[Section 8.2]{H\"or89} for the definition of wave front sets on manifolds and wave front sets of sections as well as to \cite[Section 8.4]{H\"or89} for the definition of the analytic wave front set. Note that \cite[Thm. 8.6.1]{H\"or89} is stated for functions on an analytic manifold but is readily seen to apply to sections of analytic complex vector bundles as well since the ellipticity condition guarantees that the wavefront sets of all components of the section $f$ in a local trivialization have empty analytic wave front set. Another aspect of \cite[Thm. 8.6.1]{H\"or89} is that in the proof the radius of convergence of the Taylor series of the components of the section $f$ only depends on $P$ and $Pf$. In particular, for every point $x\in M$, there exists an open neighbourhood $U$ and a local trivialization around $U$ such that the Taylor series of all solutions $f$ of $Pf=0$ in this local trivialization converges.
}
\end{rem}
\begin{rem}\label{remEquivSec}
\rm{(a) Let $G$ be a connected semisimple Lie group. Consider the linear operators $L_X,R_X:C^\infty(G)\to C^\infty(G)$, for $X\in \g$, defined as
$$L_Xf(g):=\difftev f(\exp(-tX)g)\qand R_Xf(g):=\difftev f(g\exp(tX)).$$
The space $C^\infty(G)$ carries a $(\g\oplus\g,G\times G)$-bimodule structure defined by
$$ (X_1\oplus X_2).f:=L_{X_1}f+R_{X_2}f\qand (k_1,k_2).f:=L_{k_1}R_{k_2}f:k\mapsto f(k_1^{-1}kk_2),$$
for $f\in C^\infty(G)$. In particular, the $\g\oplus \g$-module structure defines algebra homomorphisms $L:\cU(\g)\to \End(C^\infty(G))$ and $R:\cU(\g)\to \End(C^\infty(G))$. Note that $L(\cU(\g))$ consists of right-invariant differential operators on $G$ and $R(\cU(\g))$ consists of left-invariant differential operators on $G$.\\

(b) We identify $\cU(\g)$ with left-invariant differential operators on $C^\infty(G)$ by the right-action $R:\cU(\g)\to \End(C^\infty(G))$ introduced in (a). Let $\Sigma^+\subseteq\Sigma:=\Sigma(\g,\a)$ be a positive system of restricted roots and let
$$\{Y_1,\dots,Y_m\}\subseteq \k,\quad \{H_1,\dots,H_r\}\subseteq \a\qand \{X_1^{(\alpha)},\dots,X_{m_\alpha}^{(\alpha)}\}\subseteq \p_\alpha,
$$
for $\alpha\in \Sigma^+$ and $\p_\alpha:=\p\cap (\g_\alpha+\g_{-\alpha})$, be orthonormal bases with respect to the Cartan-Killing form. Let $\cB$ denote the ortonormal basis of $\g$ one obtains as a union. Then the \textit{quadratic Casimir operator} $C_\g$ is given by
$$C_\g:=\sum_{j=1}^rH_j^2+\sum_{\alpha\in \Sigma^+}\sum_{k=1}^{m_\alpha}\left(X_k^{(\alpha)}\right)^2-\sum_{i=1}^m Y_i^2
$$
and the \textit{Laplace-Beltrami operator} $\Delta_{G/K}$ of the Riemannian symmetric space $G/K$ is given by
\begin{equation}\label{LaplBeltr}
\Delta_{G/K}=\sum_{j=1}^nH_j^2+\sum_{\alpha\in \Sigma^+}\sum_{k=1}^{m_\alpha}\left(X_k^{(\alpha)}\right)^2.
\end{equation}
That $\Delta_{G/K}$ is the Laplace-Beltrami operator of $G/K$ follows from left-invariance and the fact that it defines the Laplacian on $T_e(G/K)\cong \p$ with respect to the Cartan-Killing form. The quadratic Casimir operator for $\k$ is given by $$C_\k:=\sum_{i=1}^m Y_i^2\quad\text{and thus} \quad C_\g+2C_\k=\Delta,$$ where $\Delta\in \cU(\g)$ is the Laplacian element of $\g$ with respect to the basis introduced before. It defines an elliptic differential operator on $C^\infty(G)$ since it is left invariant and, for the corresponding dual basis $\{\mu_1,\dots,\mu_n\}\in \g^\ast$ of $\cB$, one has
$$0\neq\gamma=\sum_{j=1}^N\xi_j \mu_j\in T_eG^\ast=\g^\ast,\quad\text{one has}\quad \sigma(\Delta)(\gamma)=\sum_{j=1}^n\xi_j^2>0.
$$
Similarly one shows that $\Delta_{G/K}\in \cU(\g)^K$, identified with a left-invariant differential operator on $G/K$, defines an elliptic differential operator. \\

(c) Let $f:C^\infty(G)$ be a smooth $K$-right finite function, i.e. $V:=\spann_\C R(K)f\subseteq C^\infty(G)$ is a $K$-right invariant finite-dimensional subspace. We denote the corresponding representation of $K$ by $(\tau,V)$ and consider
$$F_\tau:G\to V^\ast,\quad g\mapsto F_\tau(g):=\ev_g
$$
which is $K$-right equivariant with respect to the dual representation $(\tau^\ast,V^\ast)$. Consider the associated vector bundle
$$\E^\ast:=G\times_K V^\ast:=(G\times V^\ast)/\sim,\qfor (g,\lambda)\sim (gk,\tau^\ast(k)^{-1}\lambda),\qfor k\in K.
$$
Note that $F_\tau$ defines a smooth section $s_f:G/K\to \E^\ast$ of $\E^\ast$ by
$$ s_f(gK):= [(gk,F_\tau(gk)]=[(gk,\tau^\ast(k)^{-1}F_\tau(g)]=[(g,F_\tau(g)],\qfor g\in G.
$$
Conversely every smooth section of $\E^\ast\epi G/K$ comes from a $K$-right equivariant smooth map $F:G\to V^\ast$ in this manner and, for $v\in V$, the map $G\to \C,\;g\mapsto F(g)(v)$ is a smooth $K$-right finite function. Since $(\tau^\ast,V^\ast)$ is a finite-dimensional representation of $K$, integrating the $\k_\C$-representation yields a holomorphic extension of $\tau^\ast$ to the simply connected covering $\wt{K}_\C$ of $K_\C\subseteq G_\C$. Observe that the bundle $G\times_K V$ is a subbundle of the holomorphic vector bundle 
$$\E_V:=\wt{\Xi}_{G_\C}\times_{\wt{K}_\C} V\epi \Xi,$$
whose holomorphic sections $\cO_{\rm sec}(\Xi,\E_V)$ are $\wt{K}_\C$-equivariant holomorphic maps $\wt{\Xi}_{G_\C}\to V$.
}
\end{rem}
\begin{lem}\label{lemCharSet}
Let $(\tau,V)$ be a finite-dimensional representation of $K$. The quadratic Casimir operator $C_\g$ acts as $\Delta_{G/K}- d\tau(C_\k)$ on the sections of the vector bundle\rm{
$$G\times_K V\epi G/K\qand \text{char}(\Delta_{G/K}- d\tau(C_\k))=\text{char}(\Delta_{G/K})=\emptyset.
$$}
\end{lem}
\begin{proof}
By identifying sections with $K$-right equivariant functions $f\in \C^\infty(G,V)^K$, the assertion follows from \begin{align*}
R(C_\g)f=&\left(\sum_{j=1}^rR(H_j)^2+\sum_{\alpha\in \Sigma^+}\sum_{k=1}^{m_\alpha}R(X_k^{(\alpha)})^2-\sum_{i=1}^m R(Y_i)^2\right)f\\
=&\left(\sum_{j=1}^rR(H_j)^2+\sum_{\alpha\in \Sigma^+}\sum_{k=1}^{m_\alpha}R(X_k^{(\alpha)})^2\right)f-d\tau\left(\sum_{i=1}^m Y_i^2\right)f
\end{align*}
and Equation (\ref{LaplBeltr}) which implies that $R(C_\g)f=R(\Delta_{G/K})f-d\tau(C_\k)f$.
\end{proof}
The following proposition, for linear simple Lie groups, can be found in \cite{KSch09} as part of the proof of \cite[Thm. 3.2]{KSch09}. Here one has to address a few of the technical details, one picks up by dropping that $G$ is linear.
\begin{prop}\label{corHolExt}
Let $(\tau,V)$ be a finite-dimensional representation of $K$. Then, for every smooth $Z(\g)$-finite section $s:G/K\to G\times_K V$, there exists a $G$-left and $\wt{K}_\C$-right invariant neighbourhood $U_0$ of $G\subseteq \wt{\Xi}_{G_\C}$ in the principal crown bundle such that $s$ extends to a holomorphic section $$s:\Xi\supseteq U_0/{\wt{K}_\C}\to U_0\times_{\wt{K}_\C} V.$$
\end{prop}
\begin{proof}
Since $s$ is $Z(\g)$-finite, it follows that $\C[C_\g]s$ is finite-dimensional and thus there exists a polynomial $0\neq p\in \C[x]$ such that $p(C_\g)s=0$. Let $k$ be the degree of $p$. Then the principal symbol of $p(C_\g)$ is given by the principal symbol of $C_\g^k$. In particular, Lemma \ref{lemCharSet} implies $\text{char}(p(C_\g))=\text{char}(\Delta_{G/K})=\emptyset$ and thus the analytic Elliptic Regularity Theorem implies that the section $s:G/K\to G\times_K V$ is analytic (cf. Remark \ref{remAnalyticReg}). In particular, it follows from Remark \ref{remAnalyticReg} that there exists an open polydisc $U_\p\subseteq \p+i\Omega_\p$ around $0$ which only depends on $p(C_\g)$ and not the section $s$, for which $$\varphi:U_\p\to \Xi,\quad x+iy\mapsto [(\exp(x),iy)],\qfor x,y\in \p\qsuchthat x+iy\in U_\p$$
is a biholomorphism onto an open subset and that the Taylor series of $f\circ \varphi:U_\p \to V$ converges on $U_\p$, where $f:G\to V$ is the $K$-equivariant function defining the section $s$. Since $U_\p$ is a polydisc, this implies that $f\circ \varphi$ extends to a holomorphic function on $U_\p$ which coincides with its Taylor series. Consider the left shifts $$\lambda_{g_0}(s)(gK):=s(g_0^{-1}gK),\qfor g,g_0\in G.$$ 
Since $C_\g$ acting on $\Gamma^\infty(G/K,G\times_K V)$ is a left-invariant differential operator, one has $p(C_\g)\lambda_g(s)=0$, for all $g\in G$. Since $U_\p$ was independent of the section $s$, the same arguments imply that the Taylor series of $\lambda_g(f)\circ \varphi$ converges on $U_\p$, for all $g\in G$, i.e. $\lambda_g(f)\circ \varphi$ extends to a holomorphic function on $U_\p$.\\

We now argue that by possibly shrinking $U_\p$ and defining $U_0:=G\,\exp(U_\p)\wt{K}_\C$, we obtain a holomorphic $\wt{K}_\C$-equivariant extension of $f$ to $U_0$ by mapping
$$g\,\exp(x)k\mapsto \tau(\wt{k})^{-1}(\lambda_g(f)\circ \varphi)(x), \qfor (g,x,k)\in G\times U\times \wt{K}_\C.
$$
We shrink $U_\p$ such that $U_\p$ is $K$-invariant. In particular, $U_0\subseteq \wt{\Xi}_{G_\C}$ is open since $G\times i\Omega_\p\times \wt{K}_\C\to \wt{\Xi}_{G_\C}$, $(g,x,k)\mapsto g\,\exp(x)k$ is a quotient map and thus open. Hence,
$$G\times U_\p\times \wt{K}_\C\to V,\quad (g,x,k)\mapsto \tau(k)^{-1}(\lambda_g(f)\circ \varphi)(x)
$$
is constant on the $K$-orbits and thus factors through $f:U_0\to V$ which is a $\wt{K}_\C$-right equivariant extension of $f:G\to V$. In particular, one obtains an extension of the section $s$ to $U_0/\wt{K}_\C\subseteq \Xi$ which is holomorphic by construction.
\end{proof} 

\begin{rem}\label{remCharVar}
\rm{For a Weyl chamber $\a^+\subseteq \a$, one considers $\Omega_\a^+:=\Omega_\a\cap \a^+$ and defines
$$\wt{\Xi}_{G_\C}^+:=G\exp(i\Omega_\a^+)\wt{K}_\C\subseteq \wt{\Xi}_{G_\C}\qand \Xi^+:=\wt{\Xi}_{G_\C}^+/\wt{K}_\C.
$$ 
Note that $\wt{\Xi}_{G_\C}^+\subseteq \wt{\Xi}_{G_\C}$ is open, dense and $$G\times i\Omega_\a^+\times \wt{K}_\C\to \wt{\Xi}_{G_\C}^+,\quad (g,iH,\wt{k})\mapsto g\,\exp(iH)\wt{k}$$
is a real diffeomorphism (cf. \cite[Lemma 3.1]{KSch09}). Recall that the extension $$G\times i\Omega_\a\times \wt{K}_\C\to \wt{\Xi}_{G_\C}$$
of this map above is not everywhere regular and only defines a branched covering of $\wt{\Xi}_{G_\C}$. This is not sufficient for our purposes as we want to apply methods from the theory of Partial Differential Equations. The differential operator $C_\g$ on $\Gamma^\infty(G/K,G\times_K V)$ extends to $\Gamma^\infty(\Xi^+,\E_V)$ with holomorphic coefficient functions. To see this, we first note that the complex linear extension
$$\p\to \mathcal{V}(\Xi),\quad (\wt{X}\cdot f)(xK_\C):=\difftev f(\exp(-tX)xK_\C),$$
for $X\in \p, x\in \Xi_{G_\C}, f\in C^\infty(\Xi)$ defines an embedding of $\p_\C$ into the vector fields on $\Xi$. In particular, if $X_1,\dots,X_p$ is a basis of $\p$, then $$\wt{X}_1\vert_{xK_\C},\dots,\wt{X}_p\vert_{xK_\C}\subseteq T_{xK_\C}\Xi\quad\text{is a basis, for}\quad xK_\C\in \Xi^+.$$
(cf. \cite[Section 4]{KS05}). Another identity proved there, is that there are elements $Y_k^\gamma\in (\g_\gamma+\g_{-\gamma})\cap \k$ such that
\begin{equation}\label{Xkgamma}
X_k^\gamma=\frac{1}{\cosh(\gamma(a))} \Ad(a)^{-1}X_k^\gamma-\tanh(\gamma(a))Y_k^\gamma,
\end{equation}
for $\gamma\in \Sigma^+, 1\leq k\leq m_\gamma$ and $a\in A\exp(i\Omega_\a^+)$. For  $F\in \Gamma^\infty(\Xi^+,\E_V)$, one has
$$\left.\frac{d^2}{dt^2}\right\vert_{t=0}F(a\exp(tH_j))=(\wt{H}_j^2\cdot F)(a),\quad  \left.\frac{d^2}{dt^2}\right\vert_{t=0}F(a\exp(tY_i))=d\tau(Y_i)^2F(a).
$$
It follows from $T_e\exp=\id_\g$, that $\difftev \exp(t(X+Y))=\difftev \exp(tX)\exp(tY)$, for $X,Y\in \g$. Using this, for the first equality, one obtains
\begin{align*}
&\left.\frac{d}{dt}\right\vert_{t=0}F(a\exp(tX_k^\gamma))\\
\overset{(\ref{Xkgamma})}{=}&\left.\frac{d}{dt}\right\vert_{t=0}F\left(a\exp\left(t\frac{1}{\cosh(\gamma(a))} \Ad(a)^{-1}X_k^\gamma\right)\exp\left(-t\tanh(\gamma(a))Y_k^\gamma\right)\right)\\
=&\left.\frac{d}{dt}\right\vert_{t=0}\tau\left(\exp\left(t\tanh(\gamma(a))Y_k^\gamma\right)\right)F\left(\exp\left(t\frac{1}{\cosh(\gamma(a))} X_k^\gamma\right)a\right)\\
=&\frac{1}{\cosh(\gamma(a))}\left(\wt{X}_k^\gamma\cdot F\right)(a)+\tanh(\gamma(a))d\tau(Y_k^\gamma)F(a).
\end{align*}
In particular, the principal symbol of the holomorphic extension of $C_\g$ to a differential operator on $\Gamma^\infty(\Xi^+,\E_V)$ at $a\in A\exp(i\Omega_\a^+)$ is given by
$$\sum_{j=1}^r \wt{H}_j^2+\sum_{\gamma\in \Sigma^+}\sum_{k=1}^{m_\gamma} \frac{1}{\cosh(\gamma(a))^2}\left(\wt{X}_k^\gamma\right)^2.
$$ 
This immediately implies the following corollary, which is our main conceptual contribution to the generalization of the results of \cite{KSch09}.
}
\end{rem}
\begin{cor}
The characteristic variety of $C_\g$ acting on $\Gamma^\infty(\Xi^+,\E_V)$ does not depend on $(\tau,V)$ and coincides with the characteristic variety of the holomorphically extended Laplace-Beltrami operator $\Delta_{G/K}$ of $G/K$.
\end{cor}
\begin{defin}
\rm{Let $Z\subseteq \C^n$ be an open subset and $S\subseteq Z$ be a $C^1$-hypersurfaces and let $$P(z,D)=\sum_{\alpha\leq m}a_\alpha(z)D^\alpha,\qfor a_\alpha\in \cO(Z)\qand D^\alpha=\del_{z_1}^{\alpha_1}\dots\del_{z_n}^{\alpha_n}$$
be a linear differential operator on $Z$ with analytic coefficients. We call $S$ \textit{non-characteristic} for $P$ if for all $z\in S$ and all analytic charts $(\psi,U)$ around $z_0$ the principal symbol
$$\sigma(P)(z)(\xi)=\sum_{|\alpha|=m}a_\alpha(z)\xi^\alpha,\qfor z\in Z\qand \xi\in T_zZ^\ast\cong \C^n
$$
satisfies $\sigma(P)(z_0)(\del\psi/\del z(z_0))\neq 0$, for $z_j=x_j+iy_j$ and $$\del\psi/\del z_j(z_0)=\tfrac12(\del\psi/\del x_j(z_0)-i\del\psi/\del y_j(z_0)).$$
}
\end{defin}
We recall the remaining arguments of the proof of \cite[Thm. 3.5]{KSch09}. As we covered all the necessary and minor generalizations already, the remaining proof differs not much from the original. We include it for the sake of self-containedness.
\begin{thm}\label{ThmHolExt}
Let $(\tau,V)$ be a finite-dimensional representation of $K$. Then every $Z(\g)$-finite smooth section of the associated vector bundle $G\times_K V\epi G/K$ extends to a holomorphic section of the holomorphic complex vector bundle $\wt{\Xi}_{G_\C}\times_{\wt{K}_\C} V\epi \Xi$ over the crown.
\end{thm}
\begin{proof}
Let $s:G/K\to G\times_K V$ be a $Z(\g)$-finite smooth section and let $p$ be a polynomial such that $p(C_\g)s=0$. It follows from Proposition \ref{corHolExt} that there exists a $G\times \wt{K}_\C$ invariant open neighbourhood $U_0$ of $G$ in $\wt{\Xi}_{G_\C}$ such that $s$ extends to a holomorphic section $U_0/\wt{K}_\C \to U_0\times_{\wt{K}_\C} V$.\\

We now argue as in \cite{KSch09} that $s$ extends holomorphically to $\Xi^+$. For $h_0\in \Omega_\a^+$ such that $\exp(ih_0)\in U_0$, there exists $r>0$ such that
$$T_r(h_0):=G\exp(iB_r(h_0))\wt{K}_\C\subseteq \wt{\Xi}_{G_\C}^+\cap U_0.
$$
In \cite{KSch09}, the authors argued that $\del T_r(h_0)/\wt{K_\C}\subseteq \Xi$ is non-characteristic for the holomorphic extension of the Laplace-Beltrami operator $\Delta_{G/K}$ to $\Xi^+=\wt{\Xi}_{G_\C}^+/\wt{K}_\C$ and thus also for $C_\g$ acting on $\Gamma^\infty(\Xi^+,\E_V)$ by Remark \ref{remCharVar}. Since the principal symbol of a polynomial of a differential operator is just a power of the principal symbol of the differential operator, it follows that $\del T_r(h_0)/\wt{K_\C}\subseteq \Xi$ is non-characteristic for $p(C_\g)$ as well. Since $s$ is annihilated by $p(C_\g)$, it follows from an application of \cite[Thm. 9.4.7]{H\"or89} to the components of $s$ in a local trivialization that $s$ extends holomorphically to an open neighbourhood of $T_r(h_0)/\wt{K}_\C$. In particular, $s$ extends to a holomorphic section 
$$s:\Xi^+\to \wt{\Xi}_{G_\C}^+\times_{\wt{K}_\C}V,\quad xK_\C \mapsto [x,F(x)],\qfor F:\wt{\Xi}_{G_\C}^+\to V.$$ 
holomorphic and $\wt{K}_\C$-equivariant. It remains to prove that $s$ extends to $\Xi$. Consider the function 
$$f_g:\a+i\cW.\Omega_\a^+\to V,\quad \Ad(k)(H+iH^\prime)\mapsto \tau(k)F(gk\exp(H)\exp(iH^\prime)),
$$
for $H+iH^\prime\in \a+i\Omega_\a^+,k\in N_K(\a)$, is holomorphic since $\a+i\cW.\Omega_\a^+$ is a disjoint union of $w.(\a+i\Omega_\a^+)$, for $w\in \cW$, and $F$ is holomorphic. As $F$ is holomorphic on $U_0$, there exists a zero neighbourhood $\omega_0\subseteq \a$ such that $A\exp(i\omega_0)\subseteq U_0$. In particular,
$$\a+i\omega_0\to V,\quad H+iH^\prime \mapsto F(g\exp(H)\exp (iH^\prime))
$$
is holomorphic and coincides with $f_g$ on the intersection $(\a+i\omega_0)\cap(\a+i\cW.\Omega_\a^+)$. Therefore $f_g$ extends holomorphically to their union which is a connected tube. By Bochner's Tube Theorem (cf. \cite[Thm. 2.5.10]{H\"or66}), it follows that $f_g$ extends to the convex hull $$\a+i\Omega_\a=\text{conv}((\a+i\cW.\Omega_\a^+)\cup (\a+i\omega_0))$$
of this connected tube. The map $\varphi:G\times i\Omega_\a\times \wt{K}_\C \to \wt{\Xi}_{G_\C},\; (g,iH,k)\mapsto g\,\exp (iH)k$ is a branched covering and the topology on $\wt{\Xi}_{G_\C}$ is the quotient topology with respect to this map which can be proven with the same methods as \cite[Lemma 3.1]{KSch09}. In particular,
$$F(g\exp(iH)k):=\tau(k)^{-1}f_g(iH),\qfor g\in G, H\in \Omega_\a,k\in \wt{K}_\C
$$
defines a continuous extension of $F$ to $\wt{\Xi}_{G_\C}$ since
$$G\times i\Omega_\a\times \wt{K}_\C \to \wt{\Xi}_{G_\C},\quad (g,iH,k)\mapsto F(g\exp(iH)k):=\tau(k)^{-1}f_g(iH)
$$
is a continuous map which is constant on the fibres of $\varphi$, i.e. factors through a continuous map $F:\wt{\Xi}_{G_\C}\to V$. Since $F$ is holomorphic on $\wt{\Xi}_{G_\C}^+$, which is open and dense in $\wt{\Xi}_{G_\C}$, the function $F$ is holomorphic on $\wt{\Xi}_{G_\C}$ by Riemann's Removable Singularity Theorem (cf. \cite[Thm. I.C.3]{GR65}) which proves the assertion.
\end{proof}
 
\end{document}